\documentclass{amsart}
\usepackage{amssymb}
\usepackage{latexsym}
\input xypic
\newtheorem{theorem}{Theorem}
\newtheorem{proposition}[theorem]{Proposition}

\newtheorem{lemma}[theorem]{Lemma}
\newtheorem{conjecture}[theorem]{Conjecture}
\newtheorem{corollary}[theorem]{Corollary}
\newtheorem{definition}[theorem]{Definition}

\def\Proof{\medskip\noindent{\bf Proof: }}

\def\Z{\mathbb{Z}}

\def\C{\mathbb{C}}

\def\Q{\mathbb{Q}}

\def\C{\mathbb{C}}

\def\N{\mathbb{N}}

\def\F{\mathbb{F}}

\def\H{\mathbb{H}}
\def\Pi{\mathbb{P}^{\infty}}

\def\qed{\hfill$\square$\medskip}

\def\Zpk{\mathbb{Z}/p^{k}}
\def\Zpk1{\mathbb{Z}/p^{k-1}}

\newcommand{\rref}[1]{(\ref{#1})}

\newcommand{\cform}[3]{\begin{array}{c}
{\scriptstyle #3}\\
#1\\
{\scriptstyle #2}\end{array}}

\newcommand{\beg}[2]{\begin{equation}\label{#1}#2\end{equation}}
\def\r{\rightarrow}

\def\H{\mathbb{H}}
\def\F{\mathbb{F}}

\def\ms{\mathcal{S}}

\def\sl2{\widetilde{SL_{2}(\Z)}}

\begin{document}
\title[The symplectic Verlinde algebras]{The symplectic Verlinde algebras
and string $K$-theory}
\author[I.Kriz and C.Westerland, contributions
by J.T.Levin]{Igor Kriz and Craig Westerland\\ with contributions by Joshua T. Levin}
\thanks{I.K. was partially supported by the NSA.  
C.W. was partially supported by the NSF under agreement DMS-0705428.  
J.T.L. was partially supported by the NSF under the REU program.}
\maketitle

\begin{abstract}
We construct string topology operations in twisted $K$-theory.  
We study the examples given by symplectic Grassmannians, computing 
$K^\tau_*(L \H P^\ell)$ in detail.  
Via the work of Freed-Hopkins-Teleman, these computations are related 
to completions of the Verlinde algebras of $Sp(n)$.  We compute 
these completions, and other relevant information about the
Verlinde algebras. We also identify the completions with the twisted 
$K$-theory of the Gruher-Salvatore pro-spectra.   
Further comments on the field theoretic nature of these constructions are made. \\
\end{abstract}

Much of the recent history of algebraic topology 
has been concerned with manifestations of ideas from mathematical physics within topology.  A stunning example is
Chas-Sullivan's theory of string topology \cite{cs}, 
which provides a family of algebraic structures analogous to conformal field theory 
on the homology $H_*(LM)$ of the free loop space $LM$ of a closed orientable
manifold $M$ (\cite{godin}).  The work of Chas and Sullivan started
an entirely new field of algebraic topology, 
and led to papers too numerous to quote. Equally interesting
as this analogy however is the fact that it is not quite precise: while the notion
of conformal field theory
is supposed to be completely self-dual, the string topology coproduct in $H_*LM$ has no
co-unit.

\vspace{3mm}

The inspiration for this paper came from two sources: one is the paper of 
Cohen and Jones \cite{cj},
generalizing string topology to an arbitrary $M$-oriented generalized cohomology theory.
The other is the work of Freed, Hopkins, and Teleman \cite{fht} which identified the
famous Verlinde algebra of a compact Lie group $G$ with its equivariant 
twisted $K$-theory.  
It follows that a completion of the Verlinde algebra is isomorphic to the (non-equivariant)
twisted $K$-theory
of $LBG$. 
In more detail, in \cite{fht}, Freed-Hopkins-Teleman establish a ring isomorphism
\beg{eintro2}{{}^G K^*_\tau(G) \cong V(\tau-h(G), G)
}
between the twisted equivariant $K$-theory of a simple, simply 
connected, compact Lie group $G$ acting on itself by conjugation 
and the Verlinde algebra of positive energy representations 
of $LG$ at level $\tau-h(G)$.  It is well known that the 
Borel construction for the conjugation action $G \times_G EG$ 
is homotopy equivalent to $LBG$.  So, using a twisted version of 
the Atiyah-Segal completion theorem due to C. Dwyer \cite{cdwyer} and 
Lahtinen \cite{lahtinen}, we may conclude
\beg{eintro3}{K^*_\tau(LBG) \cong V(\tau-h(G), G)^{\wedge}_I
}
where the Verlinde algebra is completed at its augmentation ideal.
Now a striking property of the Verlinde algebra is that it is a Poincare algebra,
which is the same thing as a $2$-dimensional topological field theory (with no asymmetry).
As far as we know, there is no analogous result involving ordinary homology:
in some sense, while the $K$-theory information should be on the level
of a modular functor of 
a conformal field theory, the field theory in homology should be on
the level of the conformal field theory itself. From CFT, one can
produce finite models in characteristic $0$ in the case of $N=2$ supersymmetry,
the $A$-model and the $B$-model. However, $N=2$-supersymmetry
corresponds to extra data beyond topology (Calabi-Yau structure). (Indeed,
Fan, Jarvis and Ruan \cite{fjr} give a construction of $A$-module TQFT, coupled with compactified
gravity, in the case of a Landau-Ginzburg model orbifolds, which
is related to the Calabi-Yau case, using Gromov-Witten theory.
We return to this in the Concluding remarks.)
Is it then possible that by considering an analogue of string topology on 
twisted $K$-theory,
we could mimic, using string topology constructions, 
some of the properties of the
Verlinde algebra, and in particular, perhaps, construct a 
self-dual topological field theory
in some sense? The purpose of this paper was to investigate this question.

\vspace{3mm}

What we found was partially satisfactory, partially not. First of all, it turns out that twisted
$K$-theory of loop spaces is quite difficult to compute, even in very simple cases. We originally thought that the question may be easier
to tackle for quotients of symplectic groups (such as quaternionic Grassmanians),
because of their relatively sparse homology. This led us to specialize to the symplectic
case (it is, of course, only an example, analogous discussions should exist for
other compact Lie groups). However, it turns out that the question is quite hard, 
and relatively little could be said beyond the case of projective spaces using our
methods. Even for $\H P^\ell$, where one has a collapsing spectral sequence,
it appears that one can only describe the exact extensions by considering, indeed,
the string topology product. 

\vspace{3mm}

Exploring the connection with the Verlinde algebra turned out to be more of a success:
indeed, we show that the product in the Verlinde algebra is connected with 
the string topology product in the twisted $K$-theory of free loop space.
A key step is the investigation of Gruher and Salvatore \cite{gs, kate}, who introduced
spaces interpolating between the loop spaces of finite and infinite Grassmannians. 
On the other hand, it turned out that the coproduct is {\em not} related in the same way
to the Verlinde algebra coproduct, and in fact exhibits the same asymmetry
as elsewhere in string topology. In particular, the ``genus stabilization'' string
topology operation is $0$, while it is always injective for the Verlinde algebra. 
The genus stabilization $T$ is interesting for the following reason:
Godin has recently shown \cite{godin} that $H_*(LM)$ admits the structure of a 
``non-counital homological conformal field theory'': this
means that $H_*(LM)$ is an algebra over the homology $H_*(\ms)$ of the Segal-Tillmann 
surface operad, built out of the classifying spaces of mapping class groups.  
In \cite{tillmann}, it is shown that for a topological algebra $A$ over $\ms$, 
the group completion of $A$ is an infinite loop space.  This group completion 
involves inverting $T$.  In the string topology setting, this 
trivializes the algebra, since $T=0$.  On the other hand,
one can show for example that 
a $K$-module whose $0$-homotopy is
the Verlinde algebra with $T$ inverted admits the structure of an 
$E_\infty$-ring spectrum (this will be discussed in a subsequent paper).
We speculate in the Concluding remarks of this paper that perhaps a 
Gromov-Witten type construction
can lead to topological field theories in this 
context, bridging the gap between both contexts.

\vspace{3mm}

Computationally, we determine the
twisted $K$-homology string product and coproduct for quaternionic projective
spaces, exhibiting that these structures contain, in a cute way, more information
than the corresponding structures in ordinary homology. Also, one gets the sense
that the connection with the Verlinde algebra makes the twisted $K$-theory
construction somehow ``smaller'' than the untwisted case (by virtue of the $d_3$ differential in the twisted Atiyah-Hirzebruch spectral sequence), although the answer is not finite. 

\vspace{3mm}
The present paper is organized as follows: In Section \ref{s1}, we will start,
as a ``warm-up case'', the computation of twisted $K$-theory of the
free loop spaces of symplectic
projective spaces and the related Gruher-Salvatore spaces. We will see
in particular that even here, the free loop space is tricky, and to resolve it
further, string topology structure is needed. This, in some sense, 
serves as a motivation for the remainder of the paper. 
In Section \ref{s2}, we will attempt to extend this calculation to the
case of symplectic Grassmanians of $Sp(n)$ with $n>1$. We will see that
the situation there is still much more complicated, and raises purely algebraic
questions about the symplectic Verlinde algebras, for example the structure
of their associated graded rings under the Atiyah-Hirzebruch filtration.
These and related algebraic questions will be treated in Sections \ref{s3},
\ref{scomp}. In Section \ref{s3}, we will specifically consider the case
of Verlinde algebras of representation level $1$, where more precise information
can be obtained. In Section \ref{scomp}, we will consider the general case.
Our results include a complete explicit computation of the Douglas-Braun number,
the completion of the symplectic Verlinde algebras, a complete computation
of $T$ for the case $n=1$ and an estimate (and precise conjecture) for the 
$n>1$ case. In Section \ref{operations_section}, we finally return to string
topology, constructing the string product in twisted $K$-theory, and computing
the string topology ring structure on the twisted $K$-theory of $L\H P^\ell$,
and as a result, we determine the extensions in its additive structure. We will
see that in some sense, the ring structure is more non-trivial and interesting
than in homology.
In Section \ref{scop}, we consider the string coproduct, show the failure of
stabilization, and determine the string coproduct for quaternionic projective
spaces (with a non-trivial result). Finally, in Section \ref{limit_section},
we discuss the comparison of the product with the Verlinde algebra product
via Gruher-Salvatore spaces, and
the failure of analogous behaviour in the case of the coproduct. Section \ref{sconc}
contains concluding remarks, including the speculations on Gromov-Witten theory.
In the Appendix \ref{app}, we review very briefly some foundational material
needed in this paper.

We would like to thank Chris Douglas, Nora Ganter, Nitu Kitchloo, and Arun Ram for helpful conversations on this material.

\section{The case $n=1$.}
\label{s1}

In this paper, we focus on symplectic groups.  Denote by $G_{Sp}(\ell,n)$ the symplectic Grassmanian of $n$-dimensional
$\H$-submodules of $\H^{\ell+n}$ where $\H$ is the algebra of quaternions.
By $Sp(n)$, we shall mean its compact form, which is the group of automorphisms
of the $\H$-module $\H^n$ which preserve the norm. Of course, there is
a canonical map $G_{Sp}(\ell,n)\r BSp(n)$. Consequently,
we can induce $K$-theory twistings on $LG_{Sp}(\ell,n)$
from $K$-theory twistings on $LBSp(n)$. Here $L$ denotes the free loop space.
Twistings are classified by $H^3(?,\Z)$. We have $H^3(LBSp(n),\Z)
\cong H^{3}_{Sp(n)}(Sp(n),\Z)$ where the right hand side denotes
Borel cohomology, and $Sp(n)$ acts on itself by conjugation. As reviewed in
\cite{fht}, this group is $\Z$, and there is a canonical generator $\iota$ such
that the twisting $\tau=m\iota$ corresponds to level $m-n-1$ lowest
weight twisted representations of $LSp(n)$, when this number is positive
(since $n+1$ is the dual Coxeter number of $Sp(n)$). This is the twisting
we will be considering here.

\vspace{3mm}
We begin with a concrete computation of $K^\tau_*(L \H P^\ell)$.  
We proceed in stages, beginning with a computation of the completion of the 
Verlinde algebra for $Sp(1)$.  Using this, we compute the twisted 
$K$-theory of an intermediate space $Y(\ell, 1)$ which, in 
combination with the loop product developed in section \ref{operations_section}, 
allows us to compute $K^\tau_*(L \H P^\ell)$.

The Verlinde algebra of twisting level $m\geq 3$ (loop group representation
level $m-2$) is isomorphic to
\beg{e1}{V_m=V(m,1)=\Z[x]/Sym^{m-1}(x)
}
where $x$ is the tautological representation of $Sp(1)=SU(2)$ and
the polynomials $Sym^k(x)$ are defined inductively by
\beg{e2}{Sym^0(x)=1, \;Sym^1(x)=x, \; xSym^k(x)=Sym^{k+1}(x)+Sym^{k-1}(x).}

The augmentation ideal $I$ of $V_m$ is generated by $x-2$.  If we change variables to $y=x-2$, and define $\sigma^{m-1}(y) = Sym^{m-1}(y+2)$, then the completion of $V_m$ at $I$ is the quotient of the power series ring $\Z[[y]]$:
\beg{f1}{(V_m)^{\wedge}_I = \Z[[y]] / \sigma^{m-1}(y)
}
A straightforward induction shows that $\sigma^{m-1}(0) = Sym^{m-1}(2) = m$, so $m$ lies in the ideal $(y)$.  Consequently, equation \rref{f1} implies
\beg{f2}{(V_m)^{\wedge}_I = \prod_{p|m} \Z_p[[y]] / \sigma^{m-1}(y)}
We would like to know how many copies of $\Z_p$ appears in this completion.  We will prove a generalization of the following result in Section \ref{scomp} below. However, it is instructive to prove this special case immediately. 

\begin{proposition}
\label{p1}
Let $p$ be a prime and let $p^i||m$ (by which we mean that $p^i$ is the
$p$-primary component of $m$).  Let
$$\delta(p,m)= \left\{ \begin{array}{ll}(p^i-1)/2 &\text{if $p\neq 2$}\\
2^i-1&\text{if $p=2$}.
\end{array} \right.$$
Then the $p$-primary component of $(V_m)^{\wedge}_{I}$ is
$(\Z_p)^{\delta(p,m)}$.
\end{proposition}

Notice that $\delta(p,m)$ is the number of $2m$-th roots of unity that have positive imaginary part and also are $p$-power roots of unity.

\Proof
The polynomials
\beg{e3}{Sym^{m-1}(x)}
are Chebyshev polynomials of the
second kind applied to $2x$. This means that if we let $\zeta_m$ to be
the primitive $m$-th root of unity, then
the roots of \rref{e3} are
\beg{e4}{\zeta_{2m}^{k}+\zeta_{2m}^{-k},\;k=1,...,m-1.
}
This means that if we denote by $W$ a ring of Witt vectors at the prime
$p$ in which \rref{e3} splits, we obtain an injective homomorphism
\beg{e5}{V_{m}\otimes W\r\cform{\prod}{k=1}{m-1} W[x]/(x-
\zeta_{2m}^{k}-\zeta_{2m}^{-k})
}
with finite cokernel.
Completing this at the ideal $(x-2)$ will therefore additively give
rise to a product of as
many copies of $W$ as there are numbers $k=1,...,m-1$ such that
\beg{e6}{\zeta_{2m}^{k}+\zeta_{2m}^{-k}-2
}
has positive valuation. But now \rref{e6} is equal to
$$\zeta_{2m}^{-k}(\zeta_{2m}^{k}-1)^2,$$
so \rref{e6} has positive valuation if and only if $\zeta_{2m}^{k}-1$
does. It is standard that $\zeta_\ell-1$ has positive valuation if
and only if $\ell=p^i$ for some $i$. Indeed, sufficiency follows from
the fact that $((x+1)^{p^i}-1)/((x+1)^{p^{i-1}}-1)$ is an Eisenstein
polynomial with root $\zeta_{p^i}-1$. To see necessity, if $\zeta-1$ has
positive valuation, so does $\zeta^p-1$, so it suffices to consider the
case when $\ell$ is not divisible by $p$. But then we have a field
$\mathbb{F}_{p}[\zeta_{\ell}]$, in which $\zeta_{\ell}-1$ is invertible,
so its lift to $W$ cannot have positive valuation.
\qed

It follows from the completion theorem that the twisted $K$-theory
of a simply connected simple compact Lie group is isomorphic
to the completion of the Verlinde algebra (at a twisting
where the Verlinde algebra exists) at the augmentation ideal
with a shift equal to dimension and is equal to $0$ in dimensions
of opposite parity,
so in particular
\beg{etwee1}{K^{0}_{\tau}(L \H P^\infty)=0,
\;K^{1}_{\tau}(L \H P^\infty)\cong (V_m)^{\wedge}_{I}.
}
Next, we will compute the twisted $K$-theory of $L\H P^\ell$. We will
see that the answer is actually
quite complicated; this motivates the structure which we shall
introduce
later, and to which we will partially need to refer to complete the calculation.

\vspace{3mm}
Demonstrating another theme of this paper, we shall also see that
before considering $L\H P^\ell$, it helps to consider the ``intermediate''
space 
\beg{etwee2}{Y(\ell,1)=L\H P^{\infty}\times_{\H P^{\infty}} \H P^\ell.}
Note that \rref{etwee2} has the homotopy type of a finite-dimensional manifold: an $Sp(1)$-bundle over $\H P^\ell$. We will
see in Section \ref{s2} that $Y(\ell,1)$ is orientable
with respect to $K$-theory. Thus, its twisted $K$-homology and cohomology
with the same twisting are isomorphic, with a shift in dimensions, which
in this case is odd. 
For now, however, the most important property for us is that the twisted $K$-theory
of $Y(\ell,1)$ is easier to determine. Let us begin with a definition.
Let $i$ be such that $p^i||m$. Let $\ell$ be a positive integer.
Let $r$ be a positive integer and let $p$ be a prime. Suppose
further that $r\leq \ell+1$ and
$$2^{j-1}-1<r\leq 2^j-1$$
if $p=2$ and
$$\frac{p^{j-1}-1}{2}<r\leq \frac{p^j-1}{2}$$
if $p>2$
for some $j=1,...,i$ if $p=2$.  Then put 
$$\epsilon(p,\ell,r)= \lceil \frac{2+\ell-r}{2^{j-1}}\rceil$$
if $p=2$ and 
$$\epsilon(p,\ell,r)=\lceil\frac{2(2+\ell-r)}{(p-1)p^{j-1}}\rceil$$
if $p>2$. Let 
$$\epsilon(p,\ell,r)=0$$
in all other cases. 

\begin{theorem}
\label{t1}
We have
\beg{etwee3}{K_{1}^{\tau}Y(\ell,1)=K^{0}_{\tau}Y(\ell,1)=0
}
and
\beg{etwee4}{K_{0}^{\tau}Y(\ell,1)=K^{1}_{\tau}Y(\ell,1)=
\cform{\bigoplus}{p,r}{}\Z/p^{\epsilon(p,\ell,r)}.
}
\end{theorem}

A tool that we will need is the Serre spectral sequence in twisted $K$-theory for a fibration $F \to E \to B$ with twisting $\tau \in H^3(F)$ (and pulled back to $E$).  It takes the form $H^*(B; K^*_\tau(F)) \implies K_\tau^*(E)$.

\Proof
First, the equality between twisted $K$-homology and cohomology
follows from the fact that $Y(\ell,1)$ is an odd-dimensional, $K$-orientable
manifold.
The
canonical inclusion
$$Y(\ell,1)\r L\H P^{\infty}$$
induces a map 
\beg{etwee5}{V(m,1)^{\wedge}_{I}\r K^{1}_{\tau}Y(\ell,1).
}
Furthermore, \rref{etwee5} is a map of $K^0(\H P^\infty)=\Z[[y]]$-modules.
Clearly $y^{\ell+1}$ annihilates the target (since the right hand
side is actually a $K^0(\H P^\ell)$-module), so 
\rref{etwee5} induces a map of $\Z[[y]]$-modules
\beg{etwee6}{\Z[y]/(\sigma^{m-1}(y),y^{\ell+1})\r K^{1}_{\tau}Y(\ell,1).
}
We claim that \rref{etwee6} is actually an isomorphism. To this end,
consider the map of fibration sequences
\beg{etwee7}{
\diagram
Sp(1)\rto^= \dto & Sp(1)\dto\\
Y(\ell,1)\dto\rto & L\H P^\infty\dto \\
\H P^\ell\rto & \H P^\infty.
\enddiagram
}
Since $K_\tau^1(Sp(1)) = \Z /m$, the induced map on $E_2$-terms of
twisted $K$-cohomology Serre spectral sequences
is a $3$-fold suspension of
\beg{etwee8}{(\Z/m)[u,u^{-1}][y]\r (\Z/m)[u,u^{-1}][y]/(y^{\ell+1}).}
Here $u$ is the Bott class.  Clearly,  the spectral sequences collapse, so
\rref{etwee8}  induces an isomorphism between associated graded
objects to a (finite) filtration on \rref{etwee6}, and hence \rref{etwee6}
is an isomorphism. 

\vspace{3mm}
Thus, we are reduced to computing the $p$-completion $R(m,\ell)$
of the left hand side
of \rref{etwee6}. Let us assume that $p^i|| m$, $i\geq 1$.
The key point is to consider the Eisenstein polynomial
$$\Phi = \Phi_{p, m}=\cform{\prod}{\zeta}{}(y+2-\zeta-\zeta^{-1})$$
where the product is over all $p^i$'th roots of unity $\zeta$
with $Im(\zeta)\geq0$ such that $\zeta$ is not a $p^{i-1}$'st root of
unity. Then multiplication by $\Phi$ defines an embedding
of filtered $\Z_p[y]$-modules
\beg{etwee9a}{R(m/p,\ell)\r R(m,\ell),}
which, on the associated graded objects, is given by multiplication
by $p$ (since $\Phi = p \mod y$).
By induction on $i$, this then determines how multiples by $p$
of the generators
\beg{etwee9}{1,y,\dots,y^{\delta(p,m/p)-1}}
are represented in \rref{etwee8}:
the $p$-multiples the generators \rref{etwee9} 
are in the same filtration degree, and are represented
as $p$-multiples. Multiples of \rref{etwee9} by higher powers
$p^j$
are given by taking the image under the associated graded map of \rref{etwee9a}
of the $p^{j-1}$-multiples of the corresponding generator
\rref{etwee9}
in the associated graded object of the left hand side of \rref{etwee9a}.
Since the associated graded object of the cokernel of \rref{etwee9a}
is therefore annihilated by $p$, on the remaining generators
$$y^{\delta(p,m/p)},\dots,y^{\delta(p,m)-1}$$
multiplication by $p$ is simply represented by the reduction
modulo $p$ of $\Phi$, which is 
$$y^{\delta(p,m)-\delta(p,m/p)} = y^{p^{i - 1}/2}.$$
Recording these extensions in closed form gives the statement
of the theorem. 
\qed

It turns out that actually fully determining the twisted $K$-theory of
$L\H P^\ell$ is subtle, and seems to require full use of
the string topology product, and even then, the extensions do
not seem to follow as simple a pattern as in the case
of $Y(\ell,1)$. A complete answer will be postponed to
Section \ref{operations_section}. From a mere existence of the product,
however, we can now state the following

\begin{theorem}
\label{t1a}
\beg{et1a}{K^{\tau}_{1}(L\H P^{\ell})=K_{\tau}^{0}( L\H P^\ell)=0,
}
There exists a decreasing filtration on $K^{\tau}_{0}(L\H P^{\ell})$
such that the associated graded object is
\beg{et1b}{E_0K^{\tau}_{0}(L\H P^{\ell})=K^{\tau}_{0}(Y(\ell,1))\otimes \Z[t].
}
Similarly, there is an increasing filtration on $K_{\tau}^{1}(L\H P^\ell)$
such that the associated graded object is
\beg{et1c}{E_0K_{\tau}^{1}(L\H P^\ell)=Hom_\Z(\Z[t],K^{\tau}_{0}(Y(\ell,1))).}
\end{theorem}

\Proof
We have a canonical fibration
\beg{etweet10}{\Omega S^{4\ell+3}\r L\H P^\ell \r Y(\ell,1)
}
where the second map is the projection.
Consider the associated spectral sequence in twisted $K$-homology.
The $E^2$-term then is concentrated in even degrees, and hence
the spectral sequence collapses. 
On the other hand, the second map in \rref{etweet10}, as we will see in
Section \ref{operations_section}, is a map of rings, and in fact the entire
spectral sequence is a spectral sequence of rings. It then follows that
$K^{\tau}_{0}L\H P^\ell$ is generated, as a ring, by two generators
$y$ and $t$ where $y$ is as in \rref{etwee6}, and $t$ is the generator
of 
$$K_0\Omega S^{4\ell+3}=\Z[t].$$
Further, we see that $K^{\tau}_{0}L\H P^\ell$ satisfies the relation $y^{\ell+1}$,
and a relation which is congruent to $\sigma^{m-1}(y)$ modulo $(t)$.
The first statement follows. The second statement now follows from
the universal coefficient theorem.
\qed

\section{On the case $n>1$.}
\label{s2}

There are a number of spectral sequences we may attempt to use for
calculating the twisted $K$-theory of $LG_{Sp}(\ell,n)$ for $n>1$. Let
$W(\ell,n)$ denote the space of symplectic $n$-frames in $\H^{\ell+n}$. The
most promising spectral sequence seems to be the twisted $K$-theory
Serre spectral sequence associated with the fibration
\beg{efib1}{\Omega W(\ell,n) \r LG_{Sp}(\ell,n) \r Y(\ell,n).
}
Here
\beg{efib2}{Y(\ell,n)=LBSp(n)\times_{BSp(n)}G_{Sp}(\ell,n).
}
The reason this is advantageous is that $Y(\ell,n)$ is a homotopy equivalent to a finite-dimensional
manifold, namely a fiber bundle over $G_{Sp}(\ell,n)$ with fiber $Sp(n)$
(although not a principal bundle). Additionally, the manifolds $Y(\ell,n)$
are orientable with respect to $K$-theory because of the following
result:

\begin{lemma}
\label{lkk1}
The $K$-theory (or ordinary) homology or cohomology spectral sequences associated
to the fibrations
\beg{elkk1}{Sp(n)\r Y(\ell,n)\r G_{Sp}(\ell,n),
}
\beg{elkk2}{Sp(n)\r LBSp(n)\r BSp(n)
}
collapse to their respective $E^2$ (resp. $E_2$) terms.
\end{lemma}

\Proof
The fibration \rref{elkk1} maps in an obvious way into \rref{elkk2} which
induces an injection on $E^2$ (resp. surjection on $E_2$) terms of the spectral
sequences in question. Thus, it suffices to consider \rref{elkk2}. For the
same reason, it suffices to consider $n=\infty$ in \rref{elkk2}. But for
$n=\infty$, \rref{elkk2} is a fibration of infinite loop spaces with the
maps infinite loop maps. Since \rref{elkk2} splits, it is therefore a product
in this case, which implies the desired collapse.
\qed

We may therefore expect to use Poincar\'e duality
together with the canonical comparison map
\beg{efib3}{Y(\ell,n)\r LBSp(n)}
to help calculate the twisted $K$-theory Atiyah-Hirzebruch spectral
sequence for $Y(\ell,n)$ (see the next section for an example).

\vspace{3mm}
Unfortunately,
the twisted AHSS for $LBSp(n)$ is extremely tricky, even though
we know its target by \cite{fht}. If this calculation can
be done, though, we can solve the spectral sequence of \rref{efib1}
by the following discussion. First note that by the Eilenberg-Moore
spectral sequence, $H^*\Omega W(\ell,n)$ is the divided power algebra
on bottom generators in dimensions
$$4\ell+2, 4\ell+6,\ldots,4\ell+4n-2.$$
Since the dimensions are all even, the AHSS for $\Omega W(\ell,n)$ must
collapse (and besides, there is no twisting when $\ell>1$ by connectivity).
The spectral sequence is completely determined as the tensor product
of the twisted AHSS for $Y(\ell,n)$ and $H^*\Omega W(\ell,n)$ by the following
result.

\begin{theorem}
\label{etcoll}
For $\ell\geq n$, the twisted Serre spectral sequence of the fibration \rref{efib1}
is a spectral sequence of $H^*\Omega W(\ell,n)$-modules.
\end{theorem}

\Proof
We will use the diagonal map
from \rref{efib1} to its product with itself for the module structure.
On the product, we can take the twisting trivial on one factor,
and equal to the given twisting on the other. Thus, it suffices to
show that the {\em untwisted} $K$-theory cohomology Serre spectral
sequence associated with the fibration \rref{efib1} collapses to the
$E_2$-term. Supposed this is not the case. Since the $E_2$-term is
torsion free, the first non-trivial differential must also appear in the corresponding
$K\mathbb{Q}$-cohomology Serre spectral sequence, and hence in the ordinary
cohomology spectral sequence. So, we must prove that the ordinary rational
cohomology Serre spectral sequence associated with \rref{efib1} collapses.

To this end, first note that for $\ell>>n$, the ordinary cohomology Serre sequence
associated with the fibration
\beg{efib100}{\Omega W(\ell,n) \r \Omega G_{Sp}(\ell,n) \r Sp(n)
}
collapses. Indeed, \rref{efib100} is a principal fibration (the next term
to the right is $W(\ell,n)$) so the corresponding homology spectral sequence is
a spectral sequence of $H_*\Omega W(\ell,n)$-modules, but the elements $H_*Sp(n)$
are permanent cycles, since they have no possible target). Now consider the
Serre spectral sequence in ordinary cohomology associated with the fibration
\beg{efib101}{\Omega G_{Sp}(\ell,n)\r L G_{Sp}(\ell,n)\r G_{Sp}(\ell,n).
}
We claim that this also collapses to $E_2$. Indeed, map into the
corresponding Serre spectral
sequence \rref{elkk2}. By Lemma \ref{lkk1}, this collapses, so
it suffices to show that the polynomial generators of
$H^*(\Omega W(\ell,n),\mathbb{Q})$ are
permanent cycles in the cohomology Serre spectral sequence associated with \rref{efib101}.
For example, we may map \rref{efib1} to the case of $n=\infty$, in which case
we get
\beg{eetcol1}{\Omega Sp/Sp(\ell)\r LBSp(\ell) \r LBSp\times_{BSp}BSp(\ell).
}
By mapping into \rref{eetcol1} the case $\ell=0$, which is the fibration
\beg{eetcol2}{\Omega Sp \r *\r Sp,
}
and taking ordinary homology Serre spectral sequences, we know that for \rref{eetcol2},
the exterior generators of $H_*Sp$ transgress to the corresponding polynomial
generators of the homology of the fiber. 
This implies that in \rref{eetcol1}, the exterior generators of $H_*Sp$ of
dimensions $4\ell+3, 4\ell+7,...$ transgress and the ones in dimensions $3,...,4\ell-1$
are permanent cycles. This shows that the $E_\infty$ term of the homology
Serre spectral sequence of \rref{eetcol1} is $H_*Sp(\ell)\otimes H_*BSp(\ell)$,
all on the horizontal line. 

Now dualizing, we see that in the $H\mathbb{Q}^*$-cohomology Serre spectral
sequence of \rref{eetcol1}, the polynomial generators of
$H^*(\Omega Sp/Sp(\ell),\mathbb{Q})$ transgress. Hence, the same is true of the
polynomial generators of $H^*(\Omega W(\ell,n),\mathbb{Q})$ in the
rational cohomology Serre spectral sequence of \rref{efib1}.
What we need to show is that the transgressions of the generators
of $H^*(\Omega W(\ell,n),\mathbb{Q})$ to $H^*(Y(\ell,n),\mathbb{Q})$ are $0$.
To this end, we claim:
\beg{etclaim}{\parbox{3.5in}{For $n\leq \ell$, the map induced on homology by
the natural inclusion $Y(\ell,n)=LBSp(n)\times_{BSp(n)}G_{Sp}(\ell,n)\r LBSp
\times_{BSp}BSp(\ell)$ factors through the map induced in homology
by the natural inclusion $LBSp(\ell)\r LBSp\times_{BSp}BSp(\ell)$.}
}
Note that if \rref{etclaim} is proved, then the same statement is true
on rational cohomology by the universal coefficient theorem. Then,
the images of the transgressions of the generators of $H^*(\Omega Sp/Sp(\ell),\mathbb{Q})$
map to $0$ in $H^*(Y(\ell,n),\mathbb{Q})$, since they certainly map to
$0$
in $H^*(LBSp(\ell),\mathbb{Q})$ by the edge map theorem. This will conclude
the proof of our theorem.

\vspace{3mm}
To prove \rref{etclaim}, note that we can further compose with the
map $LBSp
\times_{BSp}BSp(\ell)\r LBSp$, since this map is injective on homology.
But the natural inclusion
$$Y(\ell,n)=LBSp(n)\times_{BSp(n)}G_{Sp}(\ell,n)\r LBSp$$
certainly factors through the inclusion
\beg{elbspk}{LBSp(n)\r LBSp.}
While note that
this is induced in our setup by a map $Sp(n)\r Sp(\ell)$ induced by
a quaternion-linear map $\H^n\r \H^\infty$ with the set of coordinates
disjoint from those involved in the inclusion $\H^\ell\r\H^\infty$
involved in the inclusion $Sp(\ell)\r Sp$ which induces the map $LBSp(\ell)\r LBSp$
involved in the statement of \rref{etclaim}, nevertheless
we have $n\leq \ell$, so the map \rref{elbspk} factors through a map
induced by {\em some} inclusion $Sp(\ell)\r Sp$ induced by inclusion
of coordinates, and any two such maps are homotopic as maps of group.
The statement of \rref{etclaim} follows.
\qed

\vspace{3mm}
\section{The representation level $1$ symplectic Verlinde algebra} \label{s3}

Let us consider the case of twisting level $n+2$ for $Sp(n)$ (representation level $1$).
The advantage is that in this case, we know by rank-level duality that the Verlinde
algebra is isomorphic to the Verlinde algebra for $Sp(1)$ at the same
twisting level (i.e. representation level $n$). More explicitly, the
generators of the Verlinde algebra are the fundamental (=level $1$)
representations of $Sp(n)$. To this end, we have a ``defining'' representation
of dimension $2n$, which we will for the moment denote by $x$. Then
the other fundamental representations are
\beg{ek1}{v_i=\Lambda^i(x)-\Lambda^{i-2}(x),\; i=2,...,n.
}
Note that there is a canonical contraction map $\Lambda^i\r \Lambda^{i-2}$.
It turns out that $xv_i(x)$ for $i\geq 1$ (we put $v_0=1$, $v_1=x$)
contains $v_{i-1}$ and $v_{i+1}$ as subrepresentations, and their complement
is an irreducible representation of level $2$. This gives the relation
\beg{ek2}{v_{i}x=v_{i+1}+v_{i-1}, i=1,...,n-1, \; v_{n}x=v_{n-1}.}
There are more level $2$ representations, but the corresponding relations
are redundant.

We see immediately that \rref{ek2} implies
\beg{ek3}{v_{i}=Sym^{i}(x)
}
where $Sym^i$ are the polynomials from Section \ref{s1}. The Braun-Douglas
number $d(n)=d(n+2,n)$ is the greatest common divisor of the differences of dimension
of the left and right hand side of each relation \rref{ek2}. This number
is contained in the augmentation ideal of the Verlinde algebra. Perhaps
surprisingly, it turns out to be very small, making the completion trivial
in most cases. We will, again, prove a generalization of the following Theorem
in Section \ref{scomp} below.

\begin{theorem}
\label{tk1}
When $n\geq 2$, we have $d(n)=2$ when $n=2^\ell-2$, and $d(n)=1$ else.
\end{theorem}

\Proof
Let $s=\cform{\sum}{i\geq 0}{}Sym^i(x)t^i$ be the generating series
of the polynomials $Sym^i$. Let us recall that
from the recursive relation \rref{ek2}
it follows that
$$sxt=st^2 +s-1,
$$
or
\beg{ekt1}{s=1/(t^2-tx+1).
}
We may think that by \rref{ek3}, this is identified with the generating
series for the $v_i$'s as defined by \rref{ek1}, which is
\beg{ekt2}{(1+t)^{2n}(1-t^2),}
but that is not quite right. The point is, \rref{ekt2} has non-trivial
coefficients also at $t^i$ with $i=n+2,...,2n+2$ which \rref{ekt1} misses.
The correct series equal to \rref{ekt2} is then
\beg{ekt3}{s(t) - t^{2n+2}s(t^{-1})= (1-t^{2n+4})s(t).
}
Thus, if a prime $p$ divides the Braun-Douglas number for $Sp(n)$,
representation level $1$, then
over $\F_p$,
\beg{ekt4}{t^{2n+4}-1 = (t^2-2nt +1)(1+t)^{2n+1}(t-1).
}
First let us note that for $p=2$, the right hand side is $(1+t)^{2n+4}$,
so \rref{ekt4} occurs if and only if $2n+4$ is a power of $2$. To see that
in this case, the Braun-Douglas number cannot be divisible by $4$, consider
the differences of the difference of dimension between the two sides of
\rref{ek2} for $i=1$. Then the left hand side is divisible by $4$,
while the right hand side is $n(2n-1)$, which is not divisible by $4$.

Now let us consider a prime $p\neq 2$. Let $p^j||2n+4$, let $h=(2n+4)/p^j$.
Then $h$ is relatively prime to $p$, so $\overline{\F}_{p}$ actually
has $h$ different roots of $t^h-1$. By looking at the right hand side
of \rref{ekt4}, which has at most $4$ different roots, we have $h\leq 4$.
But $h$ is divisible by $2$, so $h=2,4$. If $j=0$, it is verified by
direct computation that the Braun-Douglas number is $1$. When $j>0$,
we see that the left hand side contains at least $p$ factors of $t-1$,
while the right hand side contains at most $3$. So we would have to have
$p=3$, $j=1$. Thus, $2n+4$ is equal to $6$ or $12$. $6$ gives $n=1$,
which is excluded by assumption. The other case actually gives $6$
copies of $t-1$ on the left hand side of \rref{ekt4}, which cannot occur on the
right hand side.
\qed

\vspace{3mm}
We see therefore that the completion of the Verlinde algebra is a
profound operation which can lose information. In the case $n=2^\ell-2$,
it is also interesting to know the Atiyah-Hirzebruch filtration on
the Verlinde algebra completion. First, we already know that the
representation level $1$ Verlinde algebra for $Sp(n)$ with $n=2^\ell-1$ is
\beg{ext1}{\Z/2[x]/(x^{2^n-1}).
}
Next, we will construct polynomial generators $\gamma_1,...,
\gamma_n$ for the representation
ring $R(Sp(n))$ of $Sp(n)$ such that $\gamma_i$ is in
Atiyah-Hirzebruch filtration $i$. Obviously, for $i=1$, we
can just put
\beg{egm2}{\gamma_1=x-2n.
}
Next, we put
\beg{egm3}{\gamma_{i+1}=
\Lambda^{i+1}(x-2n+2i)-\Lambda^{i-1}(x-2n+2i), \;i=1,...,n-1.
}
The element \rref{egm3} is of filtration degree $\geq i+1$
because it vanishes when restricted to $Sym(i)$ (where it
is equal to $Sym^{i+1}(x^\prime)$, where $x^\prime$
is the bottom level $1$ representation of $Sp(i)$).

\vspace{3mm}
\begin{theorem}
\label{tg1}
The associated graded ring of the representation $1$ level
Verlinde algebra of $Sp(2^r-2)$ with respect of the
Atiyah-Hirzebruch filtration is isomorphic to
\beg{etg1}{\begin{array}{l}\Z[\gamma_1,...,\gamma_{2^{r}-2}]/(2,\gamma_2+\gamma_{1}^{2},...
,\\
\gamma_{2^{r-1}-1}+\gamma_{1}^{{2^{r-1}-1}},\gamma_{2^{r-1}+1}+
\gamma_{2^{r-1}}\gamma_1,...
,\gamma_{2^r-2}+\gamma_{2^{r-1}}\gamma_{1}^{2^{r-1}-2},
\gamma_{2^{r-1}}\gamma_{1}^{2^{r-1}-1})\end{array}
}
The element $2$ is represented by
\beg{etg2}{\gamma_{2^{r-1}}+\gamma_{1}^{2^{r-1}}.}
\end{theorem}

\Proof
Let the generating function of the $\Lambda^{i}(x)$'s
be $\lambda$. Then \rref{egm3} is equal to
\beg{egm5}{\text{the coefficient at $t^{i+1}$
of $\lambda(x)(1-t^2)/(1+t)^{2n-2i}$.}
}
But we know
$$\lambda(x)(1-t^2)=v=1/(t^2-tx+1),$$
so \rref{egm5} is equal to
\beg{egm5a}{\text{the coefficient at $t^{i+1}$
of $1/((1+t)^{2^{r+1}-4-2i}(t^2-tx+1))$.}
}
Let us first examine these polynomials $\mod 2$. First, note
that \rref{egm5a} is the coefficient at $t^{i+1}$ of
$$(1+t)^{4+2i}/(1-tx+t^2).$$
Upon expanding the denominator in the variable $t(x-2)$, we further
get that this is the sum of coefficients at $t^{i+1-j}$
of
\beg{egm6}{(1+t)^{2+2i-2j}(x-2)^j,
}
which is $\left(\begin{array}{c}2+2i-2j\\1+i-j\end{array}\right)$,
which is odd if and only if $j=i+1$. Let us also observe
that
\beg{egm6a}{\parbox{3.5in}{The coefficient at $t^{i+1-j}$ of \rref{egm6} is
$2\mod 4$ if and only if $i+1-j$ is a power of $2$.
}}
Thus, \rref{egm6a} will be the exact cases when the coefficient
of \rref{egm5a} at $x^{j}$ is $2\mod 4$, with the exception of the
case when $j=i$, in which the coefficient at $x^{i}$
is ``anomalously'' divisible by $4$ when $i$ is even.

Now let us examine the polynomials
\beg{egm7}{\gamma_{i}+\gamma_{1}^{i}.
}
We just proved that the polynomials \rref{egm7}
are relations in the Verlinde algebra $\mod 2$.
This is what we got from the polynomial $q_i$ obtained
from $\gamma_i$ by subtracting \rref{egm5a},
substituting $\gamma_1+2^{r+1}-4$ for $x$, and reducing $\mod 2$.
To proceed further, let us next look at the polynomial
\beg{egm8}{q_{2^{r-1}}.
}
Since the coefficient at $\gamma_{1}^{0}$ of
\rref{egm5a} is $2\mod 4$ and all the coefficients at
$\gamma_{1}^{j}$, $0<j<2^{r-1}$ are even,
we see that adding recursively multiples of
$q_{2^{r-1}}\gamma_{1}^{j}$, $0<j<2^{r-1}$,
we obtain a relation in the Verlinde algebra of the form
\beg{egm9}{2+\gamma_{1}^{2^{r-1}}+\text{higher filtration terms.}
}
This implies that $2$ is a relation in the associated graded ring,
and $2$ is represented by \rref{etg2}. Next, processing in the same
way $q_i$ with $1<i<2^{r-1}$ (i.e.
adding recursively multiples of
$q_{2^{r-1}}\gamma_{1}^{j}$, $0<j<i$), we get \rref{egm7} plus terms
of higher filtration degree, which shows that \rref{egm7} is a relation
in the associated graded ring. For $q_i$ with
$2^{r-1}<i\leq 2^{r}-2$,
we use \rref{egm6a}: the lowest coefficient of $q_i$ at a power of
$\gamma_1$ which is not divisible by $4$ is at $\gamma_{1}^{i-2^{r-1}}$.
Then add $\gamma_{1}^{i-2^{r-1}}q_{2^{r-1}}$ to make all coefficients
at $\gamma_{1}^{j}$, $j<i$, divisible by $4$. But $4$ is represented in
filtration at least $2^{r}$, so we obtain the relation
\beg{egm10}{\gamma_i+\gamma_{1}^{i-2^{r-1}}\gamma_{1}^{2^{r-1}}
}
in the associated graded ring as required.

Finally, the relation $Sym^{2^{r}-1}(x)$ in the Verlinde algebra
can be treated as $\gamma_{k+1}=\gamma_{2^{r}-1}$, giving the
desired relation in this case also. We now see by a counting argument
that the ring \rref{etg1} is indeed additively the associated graded abelian
group of the $\Z_2$-module $\Z_{2}^{2^{r}-1}$ with generators
in degrees $0,...,2^{r}-2$, and $2$ in degree $2^{r-1}$,
so therefore our list of relations is complete.
\qed

\vspace{3mm}
\noindent
{\bf Example:} Let us look at the lowest non-trivial case of Theorem
\ref{tg1}, $r=2$, so $n=2$. Let us first compute the differentials of the
twisted $K$-theory cohomology Serre spectral sequence of the fibration
\rref{elkk2}. We see from Theorem \ref{tg1}
that the only relation in the $E_\infty$ term is
\beg{eex}{\gamma_1\gamma_2.}
The vertical part of the spectral sequence is the twisted $K$-theory
of $Sp(2)$, which is the suspension by $3$ (an odd number) of
the $K_*/(2)$-exterior algebra on one generator $\iota$ in dimension $7$.
This is tensored with the horizontal part, which is $R=\Z[\gamma_1,\gamma_2]$.
Thus, the $E_2$=term is (an odd suspension of)
\beg{eex1}{\Z/2[\gamma_1,\gamma_2]\{1,\iota\}.}
But now the $R$-submodule of \rref{eex1} generated by $1$ must disappear
(to conform with the result of \cite{fht}), while the $R$-module
generated by $\iota$ needs the single relation \rref{eex} (times $\iota$).
This means that we must have
\beg{eex2}{d_{12}(1)=\iota\gamma_1\gamma_2.}
(The right hand side of \rref{eex2} is actually still multiplied by
an appropriate power of the Bott element.) The spectral sequence is a spectral
sequence of modules over the horizontal part by Lemma \ref{lkk1}.

Now let us consider the twisted $K$-theory Serre spectral
sequence (representation level $1$) of the fibration \rref{elkk1}.
We claim that the differential \rref{eex2} is still the only one present.
One sees in this case that the Borel words are $Sym^{\ell+i}(\gamma_1)$, $i=1,2$
homogenized by multiplying by the appropriate powers of $\gamma_2$. Thus,
by comparison with the case of \rref{elkk2}, no differentials $d_r$, $r<12$
are possible, and
the the $\iota$-multiple of the $E_{13}$ term is
\beg{eex3}{\Z/2[\gamma_1,\gamma_2]/(\gamma_1\gamma_2,\gamma_{1}^{i}, \gamma_{2}^{j})
}
where for $\ell$ even, $j=(\ell/2)+1$, $i=\ell+1$, and for $\ell$ odd, $j=(\ell+1)/2$, $i=\ell+2$.
Additionally, the multiple of $1$ is the Poincar\'e dual of \rref{eex3} times
the top element of $H^*(G_{Sp}(e\ll,2))$. But by comparison with
\rref{elkk2}, the elements \rref{eex3} are permanent cycles,
and by Poincar\'e duality (see Section
\ref{s2}), so are the corresponding multiples of $1$. Thus, the spectral
sequence collapses to $E_{13}$ in this case.

\section{More observations about the symplectic Verlinde algebras and their
completions}
\label{scomp}

In this section, let $V(m,n)$ denote the $Sp(n)$ Verlinde algebra
of level $m$. If $\tau$ denotes the cohomological twisting associated to this
level, then 
\beg{ecomp1}{K_{\tau}^{i}(LBSp(n))\cong V(m,n)^{\wedge}_{I}}
when $i\equiv n\mod 2$ where $I$ is the augmentation ideal of $RSp(n)$.
For $i\equiv n+1 \mod 2$, the left hand side of \rref{ecomp1} is $0$.
$V(m,n)$ is a quotient of the representation
ring $RSp(n)$ by the ideal generated by the irreducible
representations of level $m-n$. Explicitly, let $x_1$ be the
definining representation of $Sp(n)$ of dimension $2n$. Then
for $n\geq k\geq 2$, there is a natural contraction
\beg{ecomp2}{\Lambda^k(x_1)\r\Lambda^{k-2}(x_2)}
using the symplectic form. The map \rref{ecomp2} is onto,
and its kernel $x_k$ is an irreducible representation of $Sp(n)$. 
$x_1,...,x_n$ are precisely the irreducible representations of level $1$.

\vspace{3mm}
A description of irreducible representations of $Sp(n)$ of level $q$
is given as polynomials in the variables $x_1,...,x_n$ in \cite{fh}, Prop. 24.24.
Let 
$$A=(a_0, a_1,...)$$ 
be the sequence 
$$1,x_1,x_2,...,x_n, 0,-x_n,...,-x_1,1.$$
The unspecified values of $a_i$ are defined to be $0$. Then define
the {\em sequence $A$ bent at $i$} as the sequence
$$a_i,a_{i+1}+a_{i-1},a_{i+2}+a_{i-2},...\; .$$
Then irreducible representations of $Sp(n)$ of level $q$ correspond to
Young diagrams with exactly $q$ columns and at most $n$ rows. Let the lengths
of the columns be $\mu_1\geq \mu_2\geq...\mu_n$ (to recall, a Young diagram
is precisely such sequence of numbers where $\mu_1\leq n$). Then the corresponding
irreducible representation is, in $RSp(n)$, the determinant of the matrix
whose $i$'th row is given by the first $q$ terms of the sequence $A$ bent
at $\mu_i-i+1$.

\vspace{3mm}
In principle, the above description turns all algebraic questions
about the $Sp(n)$-Verlinde algebra into problems of commutative algebra. 
However, from this description,
it is not always easy to see what is happening, and
for general $m,n$, the algebra $V(m,n)$ and its completion are not 
completely understood. For example, there is a conjecture of Gepner
\cite{gep} that the Verlinde algebra is a global complete intersection
ring. 
Cummins \cite{bowr} exhibited a `level-rank' duality isomorphism
\beg{ecomp3}{V(m,n)\cong V(m,m-n-1).
}
The map \rref{ecomp3} interchanges rows and columns in Young diagrams, so
it sends $x_i$ to the level $i$ irreducible representation with 
$\mu_i=...=\mu_1=1$.

\vspace{3mm}
Completion of the $Sp(n)$-Verlinde algebra at the augmentation ideal of $RSp(n)$
does not preserve the level-rank. In fact, we saw an example
in Proposition \ref{p1} and Theorem \ref{tk1} above. More generally, 
let $d(m,n)$ be the greatest common divisor of the dimensions of the
irreducible representations of $Sp(n)$ of level $m-n$. 

\begin{proposition}
\label{pcomp1}
We have
\beg{ecomp4}{d(m,n)=\pm 
gcd\{\cform{\sum}{j=-m}{-1}\left(\begin{array}{c}2j+2(i-1)\\2(i-1)\end{array}
\right)|1\leq i\leq n\}.
}
(The right hand side is the number calculated by C.Douglas \cite{douglas}.)
\end{proposition}

\Proof
Denote, for the moment, the number on the right hand side of \rref{ecomp4}
by $e$. Then the twisted $K$-theory Serre spectral sequence
associated with the fibration
$$Sp(n)\r LBSp(n)\r BSp(n),$$
along with the calculation \cite{douglas}, and the fact that the filtration
on $K^0BSp(n)$ associated with the Atiyah-Hirzebruch spectral sequence
is the filtration by powers of the augmentation ideal, shows
that 
\beg{ecomp5}{\text{$E^{0}_{I(R(Sp(n))}V(m,n)$ is a $\Z/e$-module},}
so in other words 
$$d|e,$$
since the $0$-slice of the filtration of the Verlinde algebra by
$I(RS(n))$ is obtained by equating each $x_i$ to its dimension, which
gives $\Z/d$.

On the other hand, $K^{0}_{\tau}(G)$ is a module over $K^{G,0}_{\tau}(G)$
by restriction, which is a map of rings, while the augmentation ideal
of $R(G)$ maps to $0$, by considering the restriction $K^{G,0}(*)\r K^{0}(*)$,
which is the augmentation. this implies
$$e|d.$$
\qed

\vspace{3mm}

As remarked above, the completion of $V(m,n)$ is additively
a direct sum of a certain number of copies of $\Z_p$ over
primes $p$ which divide the number $d(m,n)$. Observe that by
\rref{ecomp4}, 
\beg{ecomp5a}{d(m,n+1)|d(m,n).}
By Nakayama's lemma, the number of copies of $\Z_p$ is the
same as the number of copies of $\Z/p$ in the completion
of $V(m,n)/p$. If we denote by $u_i$ the element $x_i$ minus
its dimension, then this is the same as taking the quotient
of the power series ring $\Z/p[[u_1,...,u_n]]$ by the
ideal $J_{m-n}$ generated by representations of level $m-n$.
Since this ideal has dimension $0$, some powers  $(u_i)^N$
must be in the ideal $J_{m-n}$. Since $V(m,n)/p$ itself
however is a finite-dimensional $\Z/p$-vector space, $N$
can be taken as its dimension, which is 
$\left(\begin{array}{c} m-1\\n\end{array}\right)$. Using Maple,
one can compute the Gr\"{o}bner basis of the ideal generated by
the level $m-n$ representations and $(u_{i})^N$. This was done
by J.T.Levin \cite{jtlevin} in a number of examples.

\vspace{3mm}
Eventually, we detected a pattern, which allowed us to compute
the completion of the symplectic Verlinde algebra in general,
as well as the number $d(m,n)$. The results are contained
in the following two theorems.

\begin{theorem}
\label{tcomp1}
For $n<m-1$, the completion of $V(m,n)$ at the augmentation
ideal of $RSp(n)$ is additively isomorphic to a sum of
\beg{ecomppt1}{\left(\begin{array}{c}\delta(p,m)\\n\end{array}
\right)} 
copies of $\Z_p$ over all primes $p$.
\end{theorem}

\Proof
When all symmetric polynomials of a finite collection of
algebraic integers have positive $p$-valuation, so does
each of them. Consider the maximal torus in $Sp(n)$
which is given by embedding the product of $n$ copies
of a chosen maximal torus of $Sp(1)$ via
the standard embedding
\beg{ecompp1}{Sp(1)\times...\times Sp(1)\subset Sp(n).}
If we choose a generating weight $t$ of $Sp(1)$,
\rref{ecompp1} gives generating weights $t_1,...,t_n$
of $Sp(n)$. Now the Grothendieck group of
level $1$ representations of $Sp(n)$ is easily seen to have
basis consisting of the elementary symmetric polynomials 
$$\sigma_1,...,\sigma_n$$
in 
\beg{ecompp2}{t_i+t_{i}^{-1}-2, \; i=1,...,n.
} 
By the results of Freed-Hopkins-Teleman, the Verlinde algebra,
when extended to a large enough ring of Witt vectors $W$,
injects, with a finite cokernel, into
a product of rings where in each individual factor, we
quotient out by a relation setting $\sigma_i$ equal to the 
$i$'th symmetric polynomial in the numbers $N_1,...,N_n$ obtained from
\rref{ecompp2} by setting 
\beg{ecompp3}{t_i=\zeta_{2m}^{j_i},\; 1\leq j_1<...<j_n\leq m-1.
}
Since the $\sigma_i$'s generate the augmentation ideal, 
the $p$-primary component of the completion of each of the factors
is $W$ if
\beg{ecompp4}{\text{All the numbers $N_1,...,N_n$ have positive $p$-valuation}}
and $0$ otherwise. 

Now by the above remarks, \rref{ecompp4} occurs if and only if
all the numbers obtained by plugging in \rref{ecompp3} into
\rref{ecompp2} have positive $p$-valuation. Now by recalling the argument in
the proof of Proposition \ref{p1}, this occurs if and only if each
$\zeta_{2m}^{j_i}$ is a $p^j$'th root of unity for some $j$. The number
of such combinations is \rref{ecomppt1}.
\qed

Note that in particular it follows that the $p$-component of the
completion of $V(m,n)$ is isomorphic to the $p$-component of
$V(m,n)$ if and only if $p=2$ and $m=2^r$ for some $r$. We also
have a more explicit evaluation of $d(m,n)$.

\vspace{3mm}
\begin{theorem}
\label{ttcomp2}
We have
\beg{ecompp5}{d(m,n)=n/gcd(n,K)
}
where for each prime $p$, $K$ is divisible by the largest power
of $p$ such that $\delta(p,K)<n$. 
\end{theorem}

\Proof
Despite the fact that Theorem \ref{tcomp1} implies a part of Theorem
\ref{ttcomp2} (namely, it detects when $p|d(m,n)$), we do not
have a proof along the same lines at this point. Instead, we need
to appeal to Proposition \ref{pcomp1}. Let
\beg{ecompp6}{S(m,i)=\left(\begin{array}{c}2m-1\\1\end{array}
\right)
+
\left(\begin{array}{c}2m-3\\i\end{array}
\right)
+...
+
\left(\begin{array}{c}1\\i\end{array}
\right)
}
Then by \rref{ecomp4},
\beg{ecompp7}{d(m,n)=gcd\{S(m,0),S(m,2),...,S(m,2(n-1))\}.}
Compute
\beg{ecompp8}{\begin{array}{l}
\cform{\sum}{i\geq 0}{}S(m,i)x^i=(1+x)^{2m-1}+(1+x)^{2m-3}+...+(1+x)\\
=
(1+x)\frac{(1+x)^{2m}-1}{(1+x)^2-1}=
\frac{(1+x}{x(x+2)}((1+x)^{2m}-1).
\end{array}
}
We see that the roots are 
\beg{ecompp8b}{\begin{array}{l}-1,\\
\zeta_{2m}^{k}-1,\; k=1,...,2m-1,\; k\neq m.
\end{array}
}
Now we will distinguish two cases. The first case is $p=2$. Then 
consider the polynomial
\beg{ecompp8a}{P_\ell=\cform{\prod}{j=0}{2^{\ell-1}-1}(x-(\zeta_{2^\ell}^{2j+1}-1)).}
The coefficients of \rref{ecompp8a} (with the exception of the leading 
coefficient) are divisible by $2$. We have
$deg(P_\ell)=2^{\ell-1}$. The polynomials
$P_\ell$ have no common roots. Furthermore, the polynomial $P_\ell$
divides \rref{ecompp8} precisely when $1\leq \ell\leq N$ where
$2^N||m$. Now $S(m,2(i-1))$ is a symmetric 
polynomial of degree $2m+1-2(i-1)$ in the roots
\rref{ecompp8b}.
Decompose $S(m,2(i-1))$ as a polynomial in the coefficients
of $P_1,...,P_N$. Observe that if
\beg{ecompp8c}{i<2^\ell,
}
then 
\beg{ecompp8d}{2m+1-2(i-1)> (2m+1)-2^1-...-2^s,
}
so each monomial of $S(m,2(i-1))$ considered as a polynomial in the roots
of \rref{ecompp8} contains roots of all the polynomials $P_\ell$ except,
at most, $s-1$ of them. It then follows that $S(m,2(i-1))$ is divisible
by $2^{n-\ell+1}$. On the other hand, if $i=2^s$, we see that \rref{ecompp8d}
turns into an equality, so there exists precisely one monomial of
$S(m,2(i-1))$ considered as a polynomial in the roots of \rref{ecompp8} 
which does not contain roots of $s$ of the polynomials $P_\ell$ (and
contains all the other roots). Since $S(m,0)=m$, we conclude
then that $2^{N-\ell}||S(m,2(i-1))$. This is what we were aiming to prove.

Let us now consider the case $p>2$. In this case, we must consider the
polynomials
\beg{ecompq1}{Q_\ell=\prod\{
(x-(\zeta_{p^\ell}^{j}-1))(x-(\zeta_{p^\ell}^{-1}-1))|j=1,...,p^{\ell-1},
\text{$j$ not divisible by $p^{\ell-1}$}
\}.
}
Then $deg(Q_\ell)=(p-1)p^{\ell-1}$, the polynomials $Q_\ell$ have no common
roots and $Q_1,...,Q_N$ divide \rref{ecompp8} where $p^N||m$.
Observe that when
\beg{ecompq2}{\frac{p^s-1}{2}<i,
}
then 
\beg{ecompq3}{2m+1-2(i-1)>(2m+1)-(p-1)-p(p-1)-...p^{s-1}(p-1),
}
so again each monomial of $S(m,2(i-1))$, considered as a polynomial
in the roots of \rref{ecompp8}, will 
contain roots of all the $Q_\ell$'s, except at most
$s-1$ of them. Therefore, $S(m,2(i-1))$ is divisible by $p^{N-s+1}$.

On the other hand, when \rref{ecompq2} turns into an equality, 
\rref{ecompq3} turns into an equality, and therefore there is precisely
one monomial in $S(m,2(i-1))$ considered as a polynomial in the roots of
\rref{ecompp8} which does not contain the roots of precisely $s$ of the
polynomials $Q_\ell$ (and contains all the other roots of \rref{ecompp8}).
Consequently, we can conclude that $p^{N-s}||S(m,2(i-1))$, as needed.
\qed

\vspace{3mm}
Let us collect one more result, which will come to use in the context
of the next section. Recall that the Verlinde algebra is a Poincar\'e
(=closed commutative Frobenius)
ring where the augmentation $\epsilon$ is defined by $\epsilon(1)=1$,
and $\epsilon(a)=0$ if $a$ is a label different from $1$. The only
axiom of a Poincar\'e ring $V$ (other than commutativity) is that the
map 
\beg{ecomp7}{M:V\otimes V\r \Z} 
defined by
$\epsilon(ab)$ for variables $a,b\in V$ define an isomorphism
\beg{ecomp7a}{V\cong Hom(V,\Z).} 
One then has an inverse of \rref{ecomp7a},
which can be interpreted as a map
\beg{ecomp88}{N:\Z\r V\otimes V.}
The composition \rref{ecomp88} with the triple product
is the ``$1$-loop translation operator'', which we denote
by $T$.

\begin{theorem}
\label{tcomp2}
For every $m\geq n+2$, 
\beg{ecomp8a}{det(T)\neq 0.}
In $V(m,1)$, we have
\beg{ecomp8}{det(T)=(-2)^{m-1}m^{m-3}.
}
\end{theorem}

\Proof
By the Verlinde conjecture (which is known to be
true for $Sp(n)$), the $\C$-Poincar\'e algebra $V(m,n)\otimes\C$ is
isomorphic to a product of $1$-dimensional algebras, which
are then automatically Poincar\'e algebras, which implies
that $T\otimes \C$ is always an invertible matrix, which implies 
\rref{ecomp8a}. To prove \rref{ecomp8},
we recall the formula \rref{e1}. Let $\zeta$ be the primitive 
$2m$'th root of unity. In the $m-1$ direct factors of $V(m,1)$,
we will then have
\beg{ecomp9}{x=\zeta^j+\zeta^{-j},\; j=1,...,m-1.
}
The labels are
\beg{ecomp9a}{1,Sym^{1}(x),....Sym^{m-2}(x),}
and they are self-contragredient, so the restriction $T_i$
of $T$ to the $i$'th summand \rref{ecomp9},
we have
\beg{ecomp9b}{T_i=\cform{\sum}{j=0}{m-2} Sym^{j}(\zeta^i+\zeta^{-i})^2.
}
Using the well known formula
\beg{ecomp9c}{Sym^{j}(2x)=\frac{(x+\sqrt{x^2-1})^{j+1}-(x-\sqrt{x^2-1})^{j+1}}{
2\sqrt{x^2-1}},
}
we get
\beg{ecomp9d}{Sym^{j}(\zeta^i+\zeta^{-i})=
\frac{\zeta^{i(j+1)}-\zeta^{-i(j+1)}}{\zeta^{i}-\zeta^{-i}},
}
which gives
\beg{ecomp9e}{T_i=\frac{-2m}{(\zeta^i-\zeta^{-i})^2}.
}
The determinant of $T$ is then the product of the numbers \rref{ecomp9a}
over $i=1,...,m-1$, which is \rref{ecomp8}.
\qed

It is worth noting that computer calculations using Maple suggest the conjecture
\beg{ever}{det(T)=2^{(m-1)\left(\begin{array}{c}m-3\\n-2\end{array}\right)}
m^{(m-3)\left(\begin{array}{c}m-3\\n-2\end{array}\right)}.}
We shall prove here a weaker statement.

\begin{theorem}
\label{tever}
When for a prime $p$, $P|d(m,n)$, $m>3$,
then $p|det(T)$ in $V(m,n)$.
\end{theorem}

\Proof
Let $\zeta$ be a primitive $2m$'th root of unity. Then
by \cite{fht}, $V(m,n)\otimes\overline{\Q}$ splits as a product
of Poincare algebras $V_I$ where 
$I=(1\leq i_1<...,i_n<m)$ and $V_I$ is the quotient of 
$V(m,n)\otimes \overline{Q}$ by the ideal generated by $x_i-\alpha_i$,
$i=1,...,n$ where $x_i$ are the level $1$ irreducible representations, and
$\alpha_i$ are the numbers obtained by expressing $_i$ as a polynomial
in the standard weights $t_i$, and evaluating
$$t_j=\zeta^{i_j}.$$
(Here $t_j$ correspond to choosing a maximal torus $T$ in $Sp(1)$
and then $T^n\subset Sp(1)^n\subset Sp(n)$ where the latter is
the standard embedding.)

Now $V_I$ as an $\overline{Q}$-algebra, is isomorphic to $\overline{Q}$.
To specify its structure as a Poincare algebra, one must evaluate
\beg{etever1}{e_I=\epsilon(1)
}
where $\epsilon$ is the augmentation. On the other han, in a
$1$-dimensional Poincare algebra, it is easy to check that
\beg{etever2}{T=1/\epsilon(1),
}
So 
\beg{etever3}{T_I=1/e_I
}
where $T_I$ is the $T$-operator on $V_I$. Now let us change from
$\overline{Q}$ to $\overline{Q_p}$. Then one also has the formula
$$T=\cform{\sum}{k\in K}{}x_kx_{k}^{*}$$
where $x_k$, $k\in K$, are the irreducible representations of 
representation level $\leq m-n$. This shows that, if we denote by
$v$ the $p$-valuation, then
$$v(T_I)\geq 0,$$
so , by \rref{etever3}, we have
\beg{etever4}{v(e_I)\leq 0.
}
We need to show that when 
\beg{etever5}{p|d(m,n),}
the inequality \rref{etever4} is sharp for at least one $I$. Assume therefore
\rref{etever5}, and that we have
\beg{etever4a}{v(e_I)=0
}
for all $I$. Recall now that the augmentation $\epsilon$ satisfies
\beg{etever6}{\begin{array}{l}
\epsilon(1)=1,\\
\epsilon(x_k)=0,\;x_k\neq 1,\; k\in K.
\end{array}
}
Denote by $x_{k,I}$ the number obtained from $k_k$ by plugging in
$\zeta^{i_j}$ for $t_j$ $j=1,...,n$ (using the Weyl character formula for 
$x_k$). Then 
$B=(x_k,I)$ is a square matrix, and \rref{etever6}, \rref{etever4a} signify
that the equation
\beg{etever7}{Bx=(1,0,...,0)^T
}
(the right hand side has $1$ in $k$'th position where $x_k=1$, and $0$
elsewhere) has a solution in $\overline{Z_p}$ (the solution
being $(e_I)^T$).
But now note that the ``twisted augmentation'' $\epsilon_k$ given by
\beg{etever8}{
\begin{array}{l}
\epsilon_k(x_k)=1,\\
\epsilon_k(x_\ell)=0,\; \ell\neq k
\end{array}
}
is given simply by
\beg{etever9}{\epsilon_k(x)=\epsilon(x\cdot x_{k}^{*}).
}
Therefore, {\em all}
equations
\beg{etever9a}{Bx=(0,....,0,1,0,...,0)^T
}
where $1$ is in the $k$'th position for any $k\in K$, have a solution in $\Z_p$. Therefore,
$B$ is an invertible matrix, and
\beg{etever10}{v det B=0.
}
But now $m>3$, so $1+\frac{2m}{p}<m$ (for $p>2$). Now consider
$I=(i_1<i_2...<i_n)$ where
$$\begin{array}{l}
i_1=1,\\
i_j\neq 1+\frac{2m}{p} \; \text{for any $j=1,...,n$},
\end{array}
$$
and let $J$ be obtained from $I$ by replacing $i_1$ with
$1+\frac{2m}{p}$. Since
$$v(\zeta-\zeta^{1+2m/p})>0,$$
the $I$'th and $J$'t column of $B$ are congruent modulo $v>0$, and
hence
$$v(det(B))>0,$$
which is a contradiction. For $p=2$, replace $\frac{2m}{p}$ by
$\frac{m}{2}=\frac{2m}{4}$.
\qed

\vspace{3mm}

\section{String topology operations in twisted $K$-theory: the product} \label{operations_section}

In \cite{godin}, Godin defines a family of string topology operations on $H_*(LM)$, parameterized by the homology $H_*(\Gamma_{g, n})$ of mapping class groups.  This extended Chas-Sullivan's proof in \cite{cs} that $H_*(LM)$ is a Batalin-Vilkovisky algebra.  It seems likely that analogues of these operations are present in the twisted $K$-homology $K_*^{\tau}(LM)$, for suitable choices of $\tau$.  In this section we focus on the most basic operation -- the loop product -- and compute it for $K_*^{\tau}(L \H P^n)$.

\subsection{Basic recollections}

Let $X$ be a topological space and $\tau \in H^3(X)$ a twisting.  Recall that $\tau$ defines a bundle $E_\tau = (E_i)_{i \in \Z}$ of spectra over $X$ with fibre the $K$-theory spectrum.  The twisted $K$-homology is 
$$K_n^\tau(X) = \pi_n(E_\tau / X) = \varinjlim_i \pi_{i+n} (E_i /X) $$
where $X$ is regarded as a subspace of $E_i$ via a section.

Functoriality is not as straightforward as for untwisted theories.  For any map $f:Y \to X$, there is a pullback bundle $E_{f^*(\tau)}$ over $Y$, equipped with a map to $E_\tau$ covering $f$.  Consequently there is an induced map 
$$f_*: K_n^{f^*(\tau)}(Y) \to K_n^\tau(X)$$

Cross products also require some care.  If $\sigma \in H^3(Z)$, consider the element $(\tau, \sigma) := p_1^*(\tau) + p_2^*(\sigma) \in H^3(X \times Z)$ where $p_i$ are projections onto $X$ and $Z$.  This defines a bundle of $K$-theory spectra $E_{(\tau, \sigma)}$ over $X \times Z$.  The smash product $E_\tau \wedge E_\sigma$ also defines a bundle of spectra over $X \times Z$; in this case the fibre is $K \wedge K$.  Multiplication in this ring spectrum gives a map $E_\tau \wedge E_\sigma \to E_{(\tau, \sigma)}$, which in turn defines an exterior cross product
$$\times: K_n^\tau(X) \otimes K_m^\sigma(Z) \to K_{m+n}^{(\tau, \sigma)}(X \times Z)$$

Now assume $X$ is a homotopy associative, homotopy unital $H$-space with multiplication $\mu: X \times X \to X$.  Composing the external cross product with $\mu_*$ gives the following:

\begin{lemma}

If $\tau$ is primitive; i.e., $\mu^*(\tau) = (\tau, \tau)$, then $\mu$ makes $K_*^\tau(X)$ into a ring.

\end{lemma}

\subsection{The loop product}

We will need the following, adapted from, e.g., Section 3.6 of \cite{fht}:

\begin{proposition}

Let $f: Y \to X$ be an embedding of finite codimension, with $K$-orientable normal bundle $N$ of dimension $d$.  There is an umkehr map
$$f^!: K_n^\tau(X) \to K_{n-d}^{f^*(\tau)}(Y)$$
for any $\tau \in H^3(X)$.

\end{proposition}

As usual, this is obtained from a Pontrjagin-Thom collapse and Thom isomorphism (which is just a shift desuspension since we are taking $N$ to be $K$-orientable).

Putting this together with the work of Chas-Sullivan and Cohen-Jones \cite{cs, cj}, gives us a loop product in twisted $K$-theory.  Namely, let $M^d$ be a closed smooth manifold of dimension $d$, and write 
$$LM = Map(S^1, M)$$
for the space of piecewise smooth maps $S^1 \to M$.  

For simplicity, assume that $M$ is $3$-connected; as a consequence, it is spin and hence $K$-orientable.  Let $\tau' \in H^4(M)$; then connectivity ensures that there is a well-defined transgressed class $\tau \in H^3(\Omega M)$.  Through the Serre spectral sequence $\tau$ gives rise to a well-defined class (represented by $1 \otimes \tau$) which we will also denote $\tau \in H^3(LM)$.  Connectivity implies, further, that $\tau \in H^3(\Omega M)$ is primitive.

\begin{theorem}

Assume that $M$ is $3$-connected.  Then $K^\tau_{*+d}(LM)$ is a unital, associative ring.

\end{theorem}

\begin{proof}

As usual, we consider the following commutative diagram in which the lower left square is cartesian:
$$\xymatrix{
\Omega M \times \Omega M \ar[d] & \Omega M \times \Omega M \ar[d] \ar[l]_-= \ar[r]^-\mu& \Omega M \ar[d] \\ 
LM \times LM \ar[d]_-{ev \times ev} & LM \times_M LM  \ar[d]_-{ev_\infty} \ar[l]_-{\tilde{\Delta}} \ar[r]^-{concat}& LM \ar[d]^-{ev} \\
M \times M & M  \ar[r]^-= \ar[l]_-\Delta & M
}$$
Primitivity of $\tau$ means that $\mu^*(\tau) = (\tau, \tau)$ on $\Omega M$.  Therefore $concat^*(\tau) = \tilde{\Delta}^*(\tau, \tau)$ on $LM$.

The inclusion $\tilde{\Delta}$ is of finite codimension with normal bundle $N \cong ev_\infty^*(TM)$, which is $K$-orientable, by assumption.  Therefore, we may form the composite $m:= concat_* \circ \tilde{\Delta}^! \circ \times$:
$$\xymatrix@1{
K_n^\tau(LM) \otimes K_m^\tau(LM) \ar[r]^m \ar[d]_-\times & K^\tau_{n+m-d}(LM) \\
K_{n+m}^{(\tau, \tau)}(LM \times LM) \ar[r]^-{\tilde{\Delta}^!} & K_{n+m-d}^{\tilde{\Delta}^*(\tau, \tau)}(LM \times_M LM) \ar[u]_-{concat_*} 
}$$
Then the arguments given in \cite{cj} in homology show that $m$ is an associative and unital product.

\end{proof}

This multiplication intertwines with the cup product in the untwisted $K$-theory of $M$ in the following way: $ev:LM \to M$ has a right inverse $c: M \to LM$; $c(p)$ is the constant loop at $p$.  Then $c$ induces a map $c_*: K^{c^*(\tau)}_*(M) \to K^\tau_*(LM)$. However, since $M$ is $3$-connected, $c^*(\tau) = 0$, so this is simply
$$c_*: K_*(M) \to K^\tau_*(LM)$$
If we give the $K_*(M)$ a ring structure via intersection theory (Poincar\'e dual to the cup product), this is clearly a ring homomorphism:

\begin{proposition} \label{mod_prop}

$K^\tau_*(LM)$ is a module over $K^*(M)$ via the map $c_*$.

\end{proposition}

Notice that unless $\tau = 0$, there is no class $\tau' \in H^3(M)$ with the property that $ev^*(\tau') = \tau$.  Consequently, $ev$ does not induce a map $ev_*: K^\tau_*(LM) \to K_*(M)$ which splits the target off of the source (as is the case in the untwisted setting).  Indeed, we will see in examples that the source is often torsion, while the target is often not.

\subsection{A Cohen-Jones-Yan type spectral sequence}

In \cite{cjy}, Cohen-Jones-Yan constructed a spectral sequence converging to $H_{*+d}(LM)$ as an algebra.  Combining their arguments with the Atiyah-Hirzebruch spectral sequence for twisted $K$-theory gives the following:

\begin{theorem}

Let $M$ be a $3$-connected, closed $d$-manifold, and choose $\tau' \in H^4(M)$ with associated transgressed twisting $\tau \in H^3(LM)$.  There is a left half-page spectral sequence $\{ E^r_{p, q}: -d \leq p \leq 0 \}$ satisfying:

\begin{enumerate}

\item The differentials $d^r: E^r_{p, q} \to E^r_{p-r, q+r-1}$ are derivations.

\item The spectral sequence converges to $K^\tau_{*+d}(LM)$ as an algebra (where we
use the loop product on $\Omega M$).

\item Its $E^2$ term is given by
$$E^2_{p, q} := H^{-p}(M, K^\tau_q(\Omega M))$$

\end{enumerate}

\end{theorem}

\subsection{An example: $\H P^\ell$}

Since $\H P^\ell $ is even dimensional, there is no degree shift in the ring $K^\tau_0(L\H P^\ell )$.  We let $\tau' \in H^4(\H P^\ell ) \cong \Z$ correspond to $m \in \Z$.

\begin{theorem} \label{hpn_thm}

For any $m \neq 0$, $K^\tau_*(L \H P^\ell )$ is concentrated in even degrees, and $K^\tau_0(L \H P^\ell )$, equipped with the loop product, is isomorphic to
$$\prod_{p | m} \Z_p [t, y] / (y^{\ell +1}, \sigma^{m-1}(y)-(\ell+1)y^\ell t)$$

\end{theorem}

We will prove this using the spectral sequence from the previous section.  First we need:
\begin{lemma}
\label{ltweet}
There is a ring isomorphism
$$K^\tau_*(\Omega \H P^\ell ) \cong (K_*/m)[t]$$
where $|t| = 4\ell +2$.

\end{lemma}

\Proof
We will use the twisted $K$-theory
Serre spectral sequence associated with the fibration
\beg{ehp1}{\Omega S^{4\ell+3}\r \Omega \H P^\ell \r Sp(1).
}
The $E_2$-term is
\beg{ehp2}{K^*\Omega S^{4\ell+3}\otimes \Lambda[x_3].
}
The first factor is vertical, the second is horizontal. The
torsion on the vertical factor disappears because of connectivity.
Further, using the Atiyah-Hirzebruch spectral sequence,
\beg{ehp3}{K^{*}\Omega S^{4\ell+3}=K^* \otimes H^*(\Omega S^{4\ell+3}) = K^* \otimes \Z\{t_{q(4\ell+2)}|q=0,1,2...\}^\wedge.}
(The $?^\wedge$ on the right hand side indicates that in $K$-cohomology,
we have the product-completion, i.e. a product of copies of $\Z$ rather
than a direct sum.
Actually, \rref{ehp3} is the product-completion
of a divided power algebra. No extensions are possible,
since there is no torsion. Regarding the differentials in
\rref{ehp2}, we know that $d_3$ is multiplication by the twisting
class, which (by our choice) is $mx_3$. Thus, the $E_4=E_\infty$
term is a suspension by $x_3$ of
\beg{ehp4}{K^*\Omega S^{4\ell+3}/(m).
}
Regarding extensions, first note that the element
represented by $x_3$ really is $m$-torsion, since our spectral
sequence is a spectral sequence of modules over the Atiyah-Hirzebruch
spectral sequence for the twisted $K$-theory of $Sp(1)$. Next,
\rref{ehp2} is always a spectral sequence
of modules over its untwisted analog. However, there $d_3=0$ by
the same argument, and further the $t_{q(4\ell+2)}$'s are permanent
cycles by filtration considerations (there are no elements in filtration
degrees $>3$). Thus, the element represented by $t_{q(4\ell+2)}x_3$ is
a product of the elements represented by $t_{q(4\ell+2)}$ and $x_3$
in the untwisted and twisted spectral sequences respectively,
and hence is $m$-torsion.

Additively, the result in twisted $K$-homology follows by the Universal Coefficient Theorem.  Multiplicatively, it follows from the fact that $H_*(\Omega S^{4\ell+3}) = \Z[t]$.

\qed

\noindent {\it Proof of Theorem \ref{hpn_thm}}

The previous lemma implies that the spectral sequence for $K^\tau_*(L\H P^\ell )$ has $E^2$-term given by
\beg{hpn_eq}{E^2_{*, *} = K_*[t, y] / (m, y^{\ell +1})
}
where $y$ has filtration degree $-4$.  The spectral sequence collapses at $E^2$ since it is concentrated in even degrees.  We therefore know that the ring structure is given by
\beg{etring1}{K_{0}^{\tau}(L\H P^\ell)=\Z[y,t]/(y^{\ell+1},\sigma^{m-1}(y)-tp(y,t))
}
for some polynomial $p(y,t)$, since, by Proposition \ref{ring_map_prop}, the map $\tilde{h}: L\H P^\ell \to Y(\ell, 1)$ induces a ring map in $K^\tau_*$. Now recall (say, \cite{wes05}) the ordinary homology Serre spectral sequence of the
fibration
$$\Omega\H P^\ell\r L\H P^\ell\r \H P^\ell.$$
The only differentials are
$$d(t^j y_\ell u)=(\ell+1)t^{j+1}$$
where $y_i$ denotes the generator of $H_{4i}(\H P^\ell)$.  Therefore, we can
conclude that $\Sigma^{-4\ell}H_*(L\H P^\ell,\Z)$, with its string topology multiplication,
is given by
\beg{etring10}{\Z[y,t,v]/(y^{\ell+1}, v^2,vy^\ell, (\ell+1)ty^\ell)
}
where $dim(y)=-4$, $dim(v)=-1$, $dim(t)=4\ell+2$. Now applying the twisted
$K$-homology Atiyah-Hirzebruch spectral sequence to \rref{etring10}, 
we get the twisting differentials
\beg{etring11}{d_3(vy^i)=my^{i+1}.
}
We see that the odd-dimensional subgroup of  $E^4$ is generated
by 
\beg{etring12}{q_i=\frac{\ell+1}{gcd(\ell+1,m)}t^i y^{\ell-1},\; i>0.
}
Let us consider the case of $i=1$ in 
\rref{etring12}. The lesser filtration degree part of the spectral sequence
is
\beg{etring20}{\Z/m\{y,y^2,...,y^\ell\}\oplus \Z\{1\}.
}
By mapping into the twisted $K$-homology AHSS for  $Y(\ell,1)$, we know
that no element of
\beg{etring21}{\Z/m\{1,y,y^2,...,y^\ell\}
}
can be the target of a differential. We conclude therefore that
\beg{etring22}{d_5(q_1)=N.1, \; N\neq 0.
}
for some number $N$. Additionally, recalling the elements \rref{etring21}
are not targets of differentials,
we see that we must have
\beg{etring23}{m|N.}
In fact, by the multiplicative structure, the differential \rref{etring22}
remains valid when we multiply by a power of the permanent cycle $t$,
so we see that the AHSS collapses to $E^6$. We conclude that
we must have
\beg{etring24}{(\ell+1)ty^\ell=q(y)}
for some polynomial $q(y)$.  This means that the polynomial
$(\ell+1)ty^\ell-q(y)$ must belong to the ideal generated by 
$y^{\ell+1}$ and $\sigma^{m-1}(y)+tp(y,t)$. By reducing $p(y,t)$
to degree $\leq \ell$ in the $y$ variable, this clearly implies
\beg{etring25}{p(y,t)|(\ell+1)y^\ell.
}
On the other hand, by our computation of the AHSS, the only
divisors of $(\ell+1)y^\ell$ which can have lower filtration degree
are multiples of
\beg{etring26}{gcd(m,\ell+1)y^\ell.}
But now note that any such polynomial is congruent to $\Z/m$-unit
times $(\ell+1)y^\ell$
modulo the relation $my^{\ell}$, which will be valid once we know
$p(y,t)$ is divisible by \rref{etring26}. Finally, we may change basis
by multplying $t$ by a $\Z/m$-unit, making the unit equal to $1$.
\qed

\vspace{3mm}
\noindent
{\bf Comment:} One now sees that $N=m^{2}/gcd(m,\ell+1)$.

\vspace{3mm}
We can now determine the precise additive structure of $K^{\tau}_{0}(L\H P^\ell)_{(p)}$.
This is essentially a standard exercise in abelian extensions of $p$-groups. Despite
some simplifications coming from the ring structure, there are many eventualities,
and we find it easiest to use a geometric pattern to represent the answer.

Let us, first, represent in this way the additive structure of $K^{\tau}_{0}(Y(\ell,1)_{(p)}$,
as calculated in Section \ref{s1}. Let $d$ be $(p-1)/2$ if $p>2$ and $1$ if
$p=2$. Imagine a table $T$ with $k$ rows indexed by
$0\leq i<k$ and $\ell +1$ columns indexed by $0\leq j \leq \ell$.
The field $(i,j)$ represents the element $p^i y^j$ in the AHSS. The extensions
are determined by paths (a step in the path represents multiplication by $p$).
The paths look as follows: we start from a field $(0,j)$, where
\beg{eext1}{\frac{p^a-1}{p-1}d\leq j< \frac{p^{a+1}-1}{p-1}d,\; 0\leq a<\delta(p,m).}
For future reference, we will call $j$ {\em critical} if equality arises for $j$ in \rref{eext1}.
Now in each path, proceed one field up all the way to row $k-a-1$:
$$(0,j),\; (1,j),\;...\;(k-a-1,j).$$
In the same row, now, proceed by 
$$\frac{p^{a+1}-1}{p-1}d$$
rows to the right,
$$(k-a-1,j),\; (k-a-1,j+\frac{p^{a+1}-1}{p-1}d),...$$
until we run out of columns in the table $T$. The lengths of the paths are
now the exponents of the $p$-powers which are orders of the summands of
$K^{\tau}_{0}(Y(\ell,1)_{(p)}$.
Let us introduce some terminology:
first, introduce an ordering on the fields of $T$: $(i^\prime,j^\prime)<(i,j)$
if $(i^\prime,j^\prime)$ precedes $(i,j)$ on a path (the paths are disjoint
so this creates no ambiguity). Next, by the {\em divisibility} of a field $(i,j)$
we shall mean the number of fields lesser that $(i,j)$ in our order
(clearly, this represents the exponent of the power of $p$ by which 
the element the field represents is divisible).

\vspace{3mm}
To describe $K^{\tau}_{0}(L\H P^\ell)_{(p)}$, imagine copies $T_n$ of
the table $T$, $n=0,1,2,....$. We will label the $(i,j)$-field in $T_n$ by
$(n,i,j)$, and it will represent the element $p^ia^jt^n$ in the associated graded
abelian group given by Theorem \ref{t1a}.
We will start with the disjoint union of the tables $T_n$ with their own paths.
However, now the paths (and the corresponding order) will be corrected 
as follows:   inductively in $n$, some fields in $T_n$ (all in column $\ell$)
will be deleted from paths in $T_n$ and appended to paths in $T_{n-1}$.
The procedure is this. Let $p^c||r$, $r=gcd(m,\ell+1)$. For a critical
generator $(n,0,j)$ (recall \rref{eext1}), let $\alpha_j$ be the number of
fields in its path in the highest row of $T_n$ the path reaches which are
not taken over by paths in $T_{n-1}$, minus $1$, plus $c$.
Then append to the path of $(n,0,j)$, in increasing row order, all
fields $(n+1,c+\alpha+d,\ell)$ whose divisibility in $T_{n+1}$ is
$$\leq \alpha+d+k-a, \; d=0,1,...,$$
(recall \rref{eext1} for the definition of $a$), and which were not already appended
to the paths of critical generators $(n,0,j^\prime)$ with $j^\prime<j$.

\vspace{3mm}
Note that only one set of corrections arises for each $n$, so it is easy
to determine the length of the corrected paths (which are the orders of
the generators of direct summands of $K^{\tau}_{0}(L\H P^\ell)_{(p)}$.
However, note that a number of scenarios can occur, and we know
of no simple formula describing the possible results in one step.

\vspace{3mm}
\noindent
{\bf Comment:} It is interesting to
note that by the additive extensions we calculated, 
$K^{\tau}_{*}(L\H P^\ell)$
with its string product cannot be a module over $K^{\tau}_{*}(\Omega \H P^{\ell})$
with its loop product structure. This exhibits the subtlety of the structures involved here.  However, as in Proposition 3.4 of \cite{cs}, there is a ring map $K^\tau_*(LM) \to K^\tau_{*-\dim M}(\Omega M)$ given by intersection with the subspace of based loops.  Hence $K^{\tau}_{*}(\Omega \H P^{\ell})$ is a module over $K^{\tau}_{*}(L\H P^\ell)$.

\section{The loop coproduct}
\label{scop}

The twisted string $K$-theory ring $K^\tau_{*-d}(LM)$ also admits a coalgebraic structure by reversing the role of $concat$ and $\tilde{\Delta}$ (and the associated Pontrjagin-Thom maps) in the definition of the loop product.

\subsection{The coproduct}

Notice that one may regard $concat$ as an embedding of a finite-codimension submanifold, just as for $\tilde{\Delta}$.  Specifically, there is a cartesian diagram
$$\xymatrix{
LM \times_M LM \ar[r]^-{concat} \ar[d]_-{ev_\infty} & LM \ar[d]^-{ev_{1,-1}} \\
M \ar[r]^-{\Delta} & M \times M
}$$
where $ev_{1,-1}$ evaluates a a loop both at the basepoint $1 \in S^1$, as well as the midpoint $-1 \in S^1$.  The pullback over the diagonal is the subspace of loops that agree at $\pm 1$.  By reparameterization, this space may be identified with $LM \times_M LM$, and the inclusion with the concatenation of loops.

Consequently $concat$ admits a shriek map in twisted $K$-theory, as well.  Again, using the fact that $concat^*(\tau) = \tilde{\Delta}^*(\tau, \tau)$, we may consider the composite
$$\xymatrix@1{
K^\tau_{n+d}(LM)  \ar[r]^-{concat^!} & K_{n}^{\tilde{\Delta}^*(\tau, \tau)}(LM \times_M LM) \ar[r]^-{\tilde{\Delta}_*} & K_{n}^{(\tau, \tau)}(LM \times LM)
 }$$

For many purposes this map suffices.  However, to properly define a coproduct, we must assume that the exterior cross product map $\times$ is an isomorphism; the preferred way to do this (using the K\"unneth Theorem) is to take our coefficients for $K$-theory to be in a field $\F$.  Define
$$\nu := \times^{-1} \circ \tilde{\Delta}_* \circ concat^! : K^\tau_{*+d}(LM; \F) \to K_*^\tau(LM; \F) \otimes_{K_*} K_*^\tau(LM; \F)$$

\begin{theorem}

The map $\nu$ defines the structure of a coassociative coalgebra on $K^\tau_{*-d}(LM; \F)$.

\end{theorem}

We do not expect $\nu$ to be counital.  Were that the case, $K^\tau_{*}(LM; \F)$ would be equipped with a nondegenerate trace, and thus finite dimensional.  But, as we have seen in Theorem \ref{t1}, this is not generally the case.

\subsection{The IHX relation}

The composite $\nu \circ m$ of the loop product and coproduct satisfies the same relations as in a Frobenius algebra:

\begin{proposition} \label{IHX_prop}

If we write left or right multiplication of $K^\tau_*(LM)$ on $K^\tau_*(LM) \otimes K^\tau_*(LM)$ by $\cdot$, then
$$\nu(xy) = x \cdot \nu(y) = \nu(x) \cdot y$$

\end{proposition}

\begin{proof}

Consider the diagram:
$$\xymatrix{
LM \times LM & LM \times (LM \times_M LM) \ar[l]_-{1 \times concat} \ar[r]^-{1 \times \tilde{\Delta}} & LM \times LM \times LM \\
LM \times_M LM \ar[u]^-{\tilde{\Delta}} \ar[d]_-{concat} & LM \times_M LM \times_M LM \ar[l]_-{1 \times concat} \ar[r]^-{1 \times \tilde{\Delta}} \ar[u]^-{\tilde{\Delta} \times 1} \ar[d]_-{concat \times 1} & (LM \times_M LM) \times LM \ar[u]^-{\tilde{\Delta} \times 1} \ar[d]_-{concat \times 1} \\
LM & LM \times_M LM \ar[l]_-{concat} \ar[r]^-{\Delta} & LM \times LM
}$$
Going around the top and right of the diagram, replacing wrong way maps by the associated umkehr map (i.e., ``push-pull") gives $x \cdot \nu(y)$.  Push-pull along the left and bottom gives $\nu(x y)$.  All of the squares except the lower left are cartesian, and the lower left is homotopy cartesian.  Consequently the two push-pull sequences are equal.  A similar diagram proves that $\nu(x\cdot y) = \nu(x) \cdot y$.

\end{proof}

\subsection{Stabilization over genus}

The ``1-loop translation operator" of Theorem \ref{tcomp2} is of considerable interest.  In the context of $K^\tau_{*}(LM)$, $T$ is the other composite of the coproduct and product:
$$T=m \circ \nu: K^\tau_{*}(LM) \to K^\tau_{*-2d}(LM)$$
In the full field theoretic language of string topology (which, to our knowledge, has not been constructed in the twisted $K$-theory setting), this is the operation induced by the class of a point in $B\Gamma_{1, 1+1}$, the moduli of Riemann surfaces of genus $1$ with $1$ incoming and $1$ outoing boundary.

Notice that, since we have assumed that $M$ is $K$-orientable, the tangent bundle gives an element $TM \in K^0(M)$.

\begin{definition} 

Let $E \in K^0(M)$ be the $K$-theoretic Euler class of $TM$:
$$E := \sum_{k=0}^d (-1)^k \Lambda^k(TM)$$

\end{definition}

We recall that $K^\tau_{*}(LM)$ is a module over $K^0(M)$ via cap products along constant loops (Proposition \ref{mod_prop}).

\begin{theorem}\label{thee}

In $K^\tau_{*}(LM)$, $T$ is given by cap product with the square of $E$:
$$T = E^2$$

\end{theorem}

\begin{proof}

This is essentially a consequence of Atiyah-Singer's computation of shriek maps in $K$-theory \cite{as1}.  Recall that if $i: Y \to X$ is an embedding of finite codimension with $K$-oriented normal bundle $N$, then the composite
$$i^* i_!: K^*(Y) \to K^*(Y)$$
is given by multiplication by $\sum _k(-1)^k \Lambda^k (N) \in K^0(Y)$.  Further, if $N$ can be written $N \cong i^*N'$, where $N'$ is a $K$-oriented bundle on $X$, then the reverse composite
$$i_! i^*: K^*(X) \to K^*(X)$$
is multiplication by $\sum _k(-1)^k \Lambda^k (N') \in K^0(X)$.  The same phenomena are true in twisted $K$-theory via the module structure over untwisted $K$-theory.

We now apply this to $concat$ and $\tilde{\Delta}$. The normal bundles to each of these embeddings are isomorphic to 
$$ev_\infty^*(TM) \to LM \times_M LM$$
where $ev_\infty$ evaluates at the common point of the two loops.  In both cases, $ev_\infty^*(TM)$ is a pullback:
$$\begin{array}{ccc}
ev_\infty^*(TM) \cong concat^*(ev^*(TM)) & {\rm and} & ev_\infty^*(TM) \cong \tilde{\Delta}^*(ev^*(TM) \times 0)
\end{array}$$

Write 
$$E' = \sum _k(-1)^k \Lambda^k (ev_\infty^*(TM)) \in K^0(LM \times_M LM)$$
and notice that $concat^*(ev^*E) = E'$.  Then, for $x \in K^*_\tau(LM)$, the dual map to $T$ is
\begin{eqnarray*}
T^*(x) & = & \nu^* (m^*(x)) \\
 & = & concat_! \tilde{\Delta}^* (\tilde{\Delta}_! concat^*(x)) \\
 & = & concat_! (E' \cdot concat^*(x)) \\
 & = & concat_! (concat^*(ev^*E \cdot x)) \\
 & = & (ev^* E)^2 \cdot x
\end{eqnarray*}
Translating this into $K$-homology turns the cup product into a cap product.

\end{proof}

Unfortunately, however, we have the following 

\begin{lemma}
\label{llem}
We always have $E^2=0$.
\end{lemma}

\Proof
$E^2$ is the Euler class of the Whitney sum of two copies of the tangent bundle of $M$.
However, in any bundle of dimension $>dim M$, the $0$-section can be moved off
itself by general position, and hence the Euler class is $0$. (Clearly, the same argument
shows that the product of the Euler class of the tangent bundle with any Euler
class of any positive-dimensional vector bundle is $0$ in any generalized cohomology
theory with respect to which $M$ is oriented.)
\qed

Therefore, we have a 

\begin{corollary}
We have $T=0$.
\end{corollary}
\qed

\subsection{The coproduct on $K^\tau_*(L \H P^\ell)$}

Despite these ``negative results'', we can use the method of Theorem \ref{thee}
to obtain non-trivial information about the coproduct. Let us consider the
string coproduct for $M= \H P^\ell$ with coefficients in $\Z/p$. Concretely, we use the fact that
the composition
\beg{ecopr1}{\diagram
K^{\tau}_{0}(LM\times_M LM,\Z/p)\rto^-{concat_*} &
K^{\tau}_{0}(LM,\Z/p)\rto^-{concat^!} & K^{\tau}_{0}(LM\times_M LM,\Z/p)
\enddiagram
}
is given by cap product with the pullback of the Euler class of $M$.
Using the notation of Theorem \ref{hpn_thm}, $1, t\in K^{\tau}_{0}(LM,\Z/p)$
are in the image of $concat_\sharp$, so 
we have
\beg{ecopr2}{concat^!(1)=(\ell+1)y^\ell,}
\beg{ecopr3}{concat^!(t)=(\ell+1)y^\ell t}
(since the Euler class is $(\ell+1)y^\ell$ -- recall also that the cup
product in $K^*(M)$ is Poincare dual to the string product on constant
loops). From \rref{ecopr2}, using Poincare duality, we immediately
conclude that
\beg{ecopr4}{\nu(1)=(\ell+1)y^\ell\otimes y^\ell.
}
Regarding \rref{ecopr3}, recall that the right hand side is equal to
$$\sigma^{m-1}(y).$$
Consider the following condition:
\beg{ecopr5}{\parbox{3.5in}{$c_i=coeff_{y^\ell}(\sigma^{m-1}(y))$ is not divisible
by $p$ for $i=\ell$, and is divisible by $p$ for all $i<\ell$.}
}
If the condition \rref{ecopr5} fails, then either $p|(\sigma^{m-1}(y),y^{\ell+1})$
in which case the relation on $K^{\tau}_{0}(LM,\Z/p)$ reads
$(\ell+1)y^\ell t=0$, and
\beg{ecopr6}{\nu(t)=0,
}
or there exists an $i<\ell$ such that $p$ does not divide $c_i$. Then, however,
$y^{\ell-i}\sigma^{m-1}(y)=y^\ell \mod p$, so $y^\ell=0$, so \rref{ecopr6} also
occurs.

If condition \rref{ecopr5} is satisfied, then the relation of Theorem \ref{hpn_thm}
with coefficients in $\Z/p$ reads
$$(\ell+1)y^\ell t= c_{\ell}y^\ell,$$
so the product formula implies
\beg{ecopr7}{\nu(t)=c_\ell =y^\ell \times y^\ell.
}

\begin{proposition}
\label{pcopr1}
The condition \rref{ecopr5} is equivalent to 
\beg{ecopr8}{\ell=\delta(p,m).
}
Thus, if \rref{ecopr8} holds, we have \rref{ecopr7}, else we have \rref{ecopr6}.
\end{proposition}

\Proof
The first statement follows from our calculation of $K^{\tau}_{0}(Y(\ell,1))$
in Section \ref{s1}. The second statement is proved in the above discussion.
\qed

\section{Completions and comparison with the Gruher-Salvatore prospectra} \label{limit_section}

All of the manifolds that we have considered as examples -- symplectic Grassmannians -- come in families associated to particular compact Lie groups.  Indeed, our computation of $K^\tau_*(L \H P^\ell )$ depends very much upon our computation of the completion of the Verlinde algebra for $Sp(1)$.  In this section, we make that connection more explicit.

There is a technical difficulty involved: while the manifolds we consider come in a sequence, e.g.,
$$\H P^1 \to \H P^2 \to \H P^3 \to \cdots, $$
it is not the case that the induced maps on free loop spaces preserve string topology operations (since intersection theory is contravariant, not covariant).  Nonetheless, one would like to study string topology operations on this system.

\subsection{Adjoint bundles}

Gruher-Salvatore achieve this goal using an interesting construction that approximates the sequence
$$L\H P^1 \to L\H P^2 \to L\H P^3 \to \cdots, $$
and furthermore preserves their modification of the string topology product.

To be more specific, we let $G$ denote a compact Lie group, and choose a model for $BG$.  We take as given an infinite sequence of even dimensional, closed, $K$-oriented manifolds $B_\ell  G$, equipped with a commutative diagram of embeddings
$$\xymatrix{
\cdots \ar[r]^-{i_{\ell -1}} & B_\ell  G \ar[d]_-{j_\ell } \ar[r]^-{i_\ell } & B_{\ell +1} G \ar[dl]^-{j_{\ell +1}} \ar[r]^-{i_{\ell +1}} & \cdots \\
& BG
}$$
with the property that the connectivity of $j_\ell $ increases with $\ell $.  Consequently the induced map
$$j: \varinjlim B_\ell  G \to BG$$ 
is a homotopy equivalence.  Let $E_\ell  G = j_\ell ^*(EG)$ be the pullback of the universal principal $G$-bundle over $BG$, with quotient $B_\ell G = E_\ell G/G$.  Lastly, define the \emph{adjoint bundle}
$$Ad(E_\ell G) := G \times_G E_\ell G,$$
where $G$ acts on itself by conjugation; this is a $G$-bundle over $B_\ell G$.

For instance, if $G = U(1)$, we may take 
$$\begin{array}{ccc}
E_\ell G = S^{2\ell +1} \subseteq \C^{\ell +1} \setminus \{ 0 \} & {\rm and} & B_\ell G = \C P^\ell 
\end{array}$$
In this case, since $G$ is abelian, $Ad(E_\ell G) = \C P^\ell  \times U(1)$.

Alternatively, if $G = Sp(n)$, we may proceed via the manifolds described in sections \ref{s1} and \ref{s2}.  That is, we may take $B_\ell (Sp(n))$ to be the symplectic Grassmannian 
$$B_\ell (Sp(n)) = G_{Sp}(\ell , n)$$
of $n$-dimensional $\H$-submodules of $\H^{\ell +n}$.  Then $E_\ell (Sp(n))$ is the Stiefel-manifold $W(\ell , n)$ of symplectic $\ell $-frames in $\H^{\ell +n}$, and
$$Ad(E_\ell (Sp(n)) = W(\ell , n) \times_{Sp(n)} Sp(n) = Y(\ell , n)$$
which is a nontrivial bundle over $G_{Sp}(\ell , n)$.

We are interested in $Ad(E_\ell G)$ because it provides an approximation to $LB_\ell G$ and $LBG$.  Consider the pair of fibrations
$$\xymatrix{
\Omega(B_\ell G) \ar[r] \ar[d]^-h & L B_\ell G \ar[r] \ar[d]^-{\tilde{h}} & B_\ell G \ar[d]^-= \\
G \ar[r] & Ad(E_\ell G) \ar[r] & B_\ell G
}$$
where the vertical map $h$ is the holonomy of a loop; in homotopy theoretic language, it is given by the first map in the fibration sequence
$$\xymatrix@1{\Omega(B_\ell G) \ar[r]^-{h} & G \ar[r] & E_\ell G \ar[r] & B_\ell G}$$
As $\ell $ tends to infinity in this sequence, $h$ tends to a homotopy equivalence, since $E_\ell G$ becomes increasingly connected.  Thus, as $\ell $ tends to infinity, $\tilde{h}: LB_\ell  G \to Ad(E_\ell G)$ tends to a homotopy equivalence.

\subsection{The product on $K^\tau_*(Ad(E_\ell G))$}

Gruher-Salvatore define a ring multiplication on $h_*(Ad(E_\ell G))$ for any cohomology theory with respect to which the vertical tangent bundle of $Ad(EG) \to BG$ is oriented.  The multiplication is an intermediary between the string topology product on $LB_\ell G$ and the fusion product on ${}^G K_\tau^*(G)$ -- it mixes intersection theory on $B_\ell G$ with multiplication in $G$.  Gruher extends these results in \cite{kate} to show that in fact $h_*(Ad(E_\ell G))$ is a Frobenius algebra over $h_*$ when $h_*$ is a graded field.  

One may introduce twistings to this story.  Assume now that $G$ is simply connected, so that every $\tau \in H^3(G)$ is primitive.  Denote also by $\tau$ the twistings in $H^3(L B_\ell G)$ and $H^3(Ad(E_\ell G))$ associated to $\tau$ by the Serre spectral sequence.  Then Gruher shows that $K_*^\tau(Ad(E_\ell G))$ admits the structure of a ring via the same multiplication as in \cite{gs}.

Specifically, one may define the product on $K^\tau_*(Ad(E_\ell G))$ via a push-pull construction.  There is a commutative diagram of fibrations
$$\xymatrix{
G \times G \ar[d] & G \times G \ar[d] \ar[l]_-= \ar[r]^-\mu&G \ar[d] \\ 
Ad(E_\ell G) \times Ad(E_\ell G) \ar[d] & Ad(E_\ell G) \times_{B_\ell G} Ad(E_\ell G)  \ar[d] \ar[l]_-{\tilde{\Delta}} \ar[r]^-{\tilde{\mu}}& Ad(E_\ell G) \ar[d] \\
B_\ell G \times B_\ell G & B_\ell G  \ar[r]^-= \ar[l]_-\Delta & B_\ell G
}$$
Here, $\tilde{\mu}$ is fibrewise multiplication in $G$.  Then the product is defined as $m:=\tilde{\mu}_* \circ \tilde{\Delta}^! \circ \times$.  The proof that twistings behave appropriately is the same as in the construction of the loop product in section \ref{operations_section}.

%
%
%
%
%

\begin{proposition} \label{ring_map_prop}

The map 
$$\tilde{h}_*: K^\tau_*(LB_\ell  G) \to K^\tau_*(Ad(E_\ell G))$$
is a ring homomorphism.

\end{proposition}

\begin{proof}

In the diagram
$$\xymatrix{
LB_\ell G \times LB_\ell G \ar[d]_-{\tilde{h} \times \tilde{h}} & LB_\ell G \times_{B_\ell G} LB_\ell G \ar[d]_-{h'} \ar[l]_-{\tilde{\Delta}} \ar[r]^-{concat}& LB_\ell G \ar[d]_-{\tilde{h}} \\
Ad(E_\ell G) \times Ad(E_\ell G) & Ad(E_\ell G) \times_{B_\ell G} Ad(E_\ell G) \ar[l]_-{\tilde{\Delta}} \ar[r]^-{\tilde{\mu}}& Ad(E_\ell G) \\
}$$
the right square homotopy commutes, since $h$ may be taken to be an $H$-map.  The left square is in fact Cartesian, so
\begin{eqnarray*}
\tilde{h}_*(x \cdot y) & = & \tilde{h}_* concat_* \tilde{\Delta}^! (x \times y) \\
                                    & = & \tilde{\mu}_* h'_* \tilde{\Delta}^! (x \times y) \\
                                    & = & \tilde{\mu}_* \tilde{\Delta}^! ( \tilde{h}_*(x) \times  \tilde{h}_*(y)) \\
                                    & = & \tilde{h}_*(x) \cdot  \tilde{h}_*(y)
\end{eqnarray*}

\end{proof}

This result, combined with the high connectivity of $\tilde{h}$ (for large $\ell$) can be taken as an indication that $K^\tau_*(Ad(E_\ell G))$ is an increasingly good approximation for the string topology multiplication on $K^\tau_*(LB_l G).$

\subsection{Limits}

The inclusions $E_\ell G \to E_{\ell+1} G$ are $G$-equivariant, so induce inclusions $Ad(E_\ell G) \to Ad(E_{\ell+1}G)$.  A Pontrjagin-Thom collapse for this embedding (or equivalently, Poincar\'e duality in twisted $K$-theory) defines a map
$$K^\tau_*(Ad(E_{\ell+1}G)) \to K^\tau_*(Ad(E_\ell G))$$
which does not shift degrees, since $codim(B_\ell G \subseteq B_{\ell+1} G)$ is even.  It is a consequence of \cite{gs} that this is a ring homomorphism.  So this gives rise to an inverse system of rings
$$\cdots \to  K^\tau_*(Ad(E_{\ell+1}G)) \to K^\tau_*(Ad(E_\ell G)) \to \cdots \to K^\tau_*(Ad(E_0G))$$

\begin{theorem} \label{completion_thm}

Let $h(G)$ be the dual Coxeter number of $G$.  There is a ring isomorphism
$$\varprojlim K_*^\tau(Ad(E_\ell G)) \cong V(\tau-h(G), G)^{\wedge}_I$$
where the left side is the inverse limit of the string topology ring structures on $K_*^\tau(Ad(E_\ell G))$, and the right side is the completion of the Verlinde algebra (with fusion product) at the augmentation ideal.

\end{theorem}

\begin{proof}

The inverse system of Pontrjagin-Thom collapse maps
$$\cdots \to Ad(E_{\ell+1}G)^{-TB_{\ell+1}G} \to Ad(E_\ell G)^{-TB_\ell G} \to \cdots \to Ad(E_0 G)^{-TB_0 G}$$
is Spanier-Whitehead dual to the direct system of inclusions
$$\cdots \supseteq Ad(E_{\ell+1}G) \supseteq Ad(E_\ell G) \supseteq \cdots \supseteq Ad(E_0G)$$
and so there is an isomorphism
\beg{edual}{\varprojlim K_*^\tau(Ad(E_\ell G)) \cong \varprojlim K^{-*}_\tau(Ad(E_\ell G))}

The main result of \cite{kate} is that for untwisted homology theories, this duality throws the string topology product onto the ``fusion product."  The same is true in the twisted setting: as we have seen, the multiplication on $K_*^\tau(Ad(E_\ell G))$ is given by the formula $m:=\tilde{\mu}_* \circ \tilde{\Delta}^! \circ \times$.  This is evidently Poincar\'e dual to the product
$$p: K^*_\tau(Ad(E_\ell G)) \otimes K^*_\tau(Ad(E_\ell G)) \to K^*_\tau(Ad(E_\ell G))$$
defined as the composite $p = \tilde{\mu}_! \circ \tilde{\Delta}^* \circ \times$.  Therefore (\ref{edual}) is a ring isomorphism.  

Furthermore, since $\varinjlim Ad(E_\ell G) = Ad(EG) =G \times_G EG$, a $\lim^1$ argument implies that
$$\varprojlim K^*_\tau(Ad(E_\ell G)) \cong K^*_\tau(Ad(EG)) \cong {}^G K_\tau^*(G \times EG)$$
The last is the $G$-equivariant twisted $K$-theory of $G \times EG$.  By the twisted version of the Atiyah-Segal completion theorem \cite{cdwyer, lahtinen}, this is isomorphic to ${}^G K_\tau^*(G)^{\wedge}_I$.  Combining these results gives a ring isomorphism
$$\varprojlim K_*^\tau(Ad(E_\ell G)) \cong {}^G K_\tau^*(G)^{\wedge}_I$$
where the multiplication on the right side of the isomorphism is given by the $G$-equivariant transfer $\mu_!$ to the principal $G$-bundle $\mu: G \times G \to G$.  The main theorem in \cite{fht} then gives the desired isomorphism.  That this isomorphism preserves the ring structure is immediate, as we note that \cite{fht} show that the fusion product on $V(\tau-h(G), G)$ is carried to the product on ${}^G K_\tau^*(G)$ defined as the transfer to $\mu$.

\end{proof}

We note that a completion is a form of inverse limit; namely,
$$V(\tau-h, G)^{\wedge}_I = \varprojlim V(\tau-h, G)/I^l$$
We are thus lead to wonder if the isomorphism of Theorem \ref{completion_thm} is in fact realized on a geometric level:

\begin{conjecture} 

There is an increasing function $N: \N \to \N$ and a collection of isomorphisms 
$$f_\ell: K_*^\tau(Ad(E_\ell G)) \cong V(\tau-h, G)/I^{N(\ell)}$$
which are coherent across the inverse system -- that is, they induce the isomorphism of Theorem \ref{completion_thm} upon passage to the limit.  

\end{conjecture}

This is true in the case of $G=Sp(1)$ (where $N(\ell) = \ell+1$), as evidenced by Theorem \ref{hpn_thm}.

%
%
%
%
%
%
%
%
%
%
%

\subsection{The coproduct}

Gruher also defines a coproduct on the twisted $K$-theory $K^\tau_*(Ad(E_\ell G); \F)$ with field coefficients; it is given by
$$\nu = {\times}^{-1} \circ \tilde{\Delta}_* \circ \tilde{\mu}^!: K^\tau_*(Ad(E_\ell G); \F) \to K^\tau_*(Ad(E_\ell G); \F) \otimes K^\tau_*(Ad(E_\ell G); \F)$$
For untwisted homology theories, this map is the Poincar\'e dual to the product, and in fact forms part of a Frobenius algebra structure.  It is unclear whether the same is true in the twisted setting, since the constant map $Ad(E_\ell G) \to pt$ does not induce a map in twisted $K$-homology.

Nonetheless, the Pontrjagin-Thom map
$$K^\tau_*(Ad(E_{\ell+1}G)) \to K^\tau_*(Ad(E_\ell G))$$
preserves the coproduct; consequently the inverse limit
$$\varprojlim K_*^\tau(Ad(E_\ell G)) \cong V(\tau-h(G), G)^{\wedge}_I$$
admits a coproduct induced by $\nu$.  It is not at all obvious whether this is related to the coproduct derived from the Frobenius algebra structure on the Verlinde algebra.

However, it is \emph{not} true that the map
$$\tilde{h}_*: K^\tau_*(LB_\ell  G; \F) \to K^\tau_*(Ad(E_\ell G); \F)$$
is a homomorphism of coalgebras.  One can try to mimic the proof of Proposition \ref{ring_map_prop}, but  the argument fails, since the right square is not in fact Cartesian.  Indeed, $concat_!$ is given by intersection theory, whereas $\mu_!$ is given by a $G$-transfer.

\section{Concluding remarks}
\label{sconc}

We have examined several different field theories in this paper.  The Verlinde algebra $V(m, G)$ is a Poincar\'e algebra, or equivalently, a topological quantum field theory (TQFT).  We have constructed a product and coproduct on the twisted string $K$-theory of a manifold $K^\tau_*(LM)$.  By virtue of the IHX relation (Proposition \ref{IHX_prop}, these combine in such a way as to give $K^\tau_*(LM)$ the structure of a ``TQFT without trace" or ``positive boundary TQFT," as was done in the homological setting by \cite{cg}.  The adjoint bundle $K$-theories $K^\tau_*(Ad(E_\ell G))$ serve as an imperfect bridge between these, preserving products, but not necessarily coproducts.

Despite this connection, when it comes to higher-genus operations, $V(m, G)$ and $K^\tau_*(LM)$ display markedly different behavior, as evidenced by the vanishing of $T$ in $K^\tau_*(LM)$), and its rational invertbility in $V(m, Sp(n))$.  It is natural to ask for the reasons behind these differences.  To approach this question, consider yet another form of field theory: Gromov-Witten theory.  We would like to think of the Verlinde algebra as a twisted $K$-theoretic analogue of Gromov-Witten theory for the stack $[*/G]$. 

String topology and Gromov-Witten theory share some ideas in their construction, at least on a schematic level.  They both involve a push-pull diagram of the form:
\beg{econc}{
(LX)^m \leftarrow Map(\Sigma, X) \rightarrow (LX)^n
}
where $\Sigma$ is a surface with $m+n$ boundary components, the maps are restrictions along boundaries, and the various function spaces are to be interpreted in the right categories.  

The shriek map in Gromov-Witten theory applies a type of intersection theory, using the deep fact of the existence of a virtual fundamental class on the compactification of the moduli of maps $\Sigma \to X$.  In contrast, the  one in string topology applies a Becker-Gottlieb type transfer using the fairly straightforward fact that the maps in \rref{econc} are fibrations with compact fibre when $X=BG$.  Said another way, in string topology, the fibers of the maps in \rref{econc} are compact, whereas in Gromov-Witten theory, the spaces themselves are compact (to the eyes of cohomology, at least).

As hinted in the Introduction, in this paper, we were interested in $K$-theory
information. In conformal field theory, an example of such information is the Verlinde
algebra, or more precisely modular functor of a CFT. According to modern string
theory, the physical aspects of strings, such as branes or even physical partition
function, are also based on $K$-theory. If we look at homology with characteristic $0$
coefficients instead of $K$-theory, topological quantum field theories should come
out of the ``state space'' of the conformal field theory itself. In effect, when $N=(2,2)$
supersymmetry is present, we can construct such models, namely the $A$-model
and the $B$-model. This also fits into the Gromov-Witten picture: $N=(2,2)$ supersymmetric
conformal field theories are expected to arise not from ordinary manifolds, but
from Calabi-Yau varieties. While a rigorous direct construction of such model
is not known, Fan-Jarvis-Ruan \cite{fjr} constructed the topological $A$-model,
along with TQFT structure, and coupling to compactified gravity, in the related case
of Landau-Ginzburg orbifold, via applying Gromov-Witten theory to the Witten equation. 
One can then ask if one can somehow extend these methods to obtain $K$-theory
rather than characteristic $0$ homology information.

\section{Appendix: Some foundations}

\label{app}

It is not the purpose of the present paper to 
discuss in detail the foundations of twisted $K$-theory.
Many of the results needed here can be read off
from the approach of \cite{as1}. More systemic
approaches, needed for some of the more complicated
assertions, use the setup of parametric spectra,
which is developed for example in \cite{ms}, or \cite{hu}
(the latter approach is simpler, but requires some corrections).
The main point is to note that the required foundations
exist, even though they are somewhat scattered throughout
the literature. Let us recapitulate here some key points.

\vspace{3mm}
Let us work in the category of parametric $K$-modules
over a space $X$,and let us denote this category by $K-mod/X$.
Then the projection $p:X\r*$
(actually, the point can be replaced by any space) has a pullback map 
$$p^*:K-mod/* \r K-mod_X,$$
which has a left adjoint $p_\sharp$ and a right adjoint $p_*$.

The category $K-mod/X$ has an internal (``fiberwise'')
smash product $\wedge_{K/X}$ and function spectrum $F_{K/X}$.
Twistings $\tau$ are objects of $K-mod/X$ which are invertible
under $\wedge_{K/X}$. There are the usual ``exponential'' adjunctions.

\vspace{3mm}
Now let us consider the universal coefficient theorem. It is better to deal
with $K$-modules than with coefficients. The $\tau$-twisted
$K$-homology module is $p_\sharp\tau$, the $\tau$-twisted $K$-cohomology
module is $p_*\tau$. We have
\beg{efound1}{\begin{array}{l}F_{K/*}(p_\sharp\tau,K)=p_*F_{K/*}(\tau,p^*K)=p_*F_{K/*}(\tau,1)=\\
p_*F_{K/*}(1,\tau^{-1})=p_*\tau^{-1}.
\end{array}}
So this is the usual universal coefficient theorem, the contribution of the twisting
is that it gets inverted. Note, however, that the complex conjugation 
automorphism of $K$-theory reverses the sign of twisting,
so $K$-(co)-homology groups with opposite twistings are isomorphic.

\vspace{3mm}
Let us now discuss Poincare duality:
When $X$ is a closed (finite-dimensional) manifold,
its $K$-dualizing object
$\omega$ is $K$ fiberwise smashed with its stable
normal bundle (considered as a parametric spectrum over $X$).
Up to suspension by the dimension $d$ of $X$, $\omega$ is a twisting.
When $\omega=1[-d]$, $X$ is called $K$-orientable.
(so when for example $H^3(X,Z)=0$, $X$ is $K$-orientable).
Poincare duality states that for a parametric $K$-module over $X$,
\beg{efound2}{
p_*? = p_\sharp(\omega\wedge_{K/X}?).
}
So indeed, when the manifold is orientable, its $K_\tau$ homology
and cohomology is the same up to dimensions shift. 

\vspace{3mm}
Using these foundations, usual results on $K$-theory extend immediately
to the twisted context. For example, there is a twisted Serre (and hence Atiyah-Hirzebruch)
spectral sequence converging to twisted $K$-homology or cohomology. 
Also, the above foundations can be extended to the equivariant case,
in which case there is a completion theorem, asserting that for $G$
compact Lie, and a finite
$G$-CW complex $X$, the non-equivariant twisted $K$-cohomology of
$X\times_G EG$ is isomorphic to the completion of the equivariant
twisted $K$-theory of $X$, completed at the augmentation ideal of
$R(G)=K^{*}_{G}(*)$.  For a detailed proof of this result, we refer the reader to \cite{cdwyer, lahtinen}.

\vspace{3mm}
Let us make one more remark on notation. We will be using both the induced map
and the transfer for a closed inclusion of manifolds $f:X\r Y$. 
Another form of duality then states that
\beg{edd1}{f_{!}\simeq f_*.}
Here $f_!$ is $f_\sharp$ composed with smashing with the sphere bundle which is the fiberwise
$1$-point compactification of the normal bundle of $X$ in $Y$ (alternately, take
the corresponding invertible parametrized $K$-module, and smash over $K_X$). 
Our interest in this
is that we need to consider the induced map and transfer map of a smooth
inclusion $f$.  First of all, for a twisting $\tau$ on $Y$, if we denote by
$[?,?]$ homotopy classes of maps of parametrized $K$-modules, and denote
by $1_Y$ the trivial $K$-module on $Y$, and by $\tau$ a twisting on $Y$,
clearly we have a map
\beg{edd2}{f^*:[1_Y,\tau]\r [1_X,f^*\tau],}
as $1_X=f^*1_Y$. Assuming now (say) that the normal bundle of
$X$ in $Y$ is $K$-orientable of real codimension $k$, we can produce a map in the other direction
\beg{edd3}{f_!:[1_X,f^*\tau]\r [1_Y,\tau[k]]}
by taking a map $1_X=f^*1_Y\r f^*\tau$, taking the adjoint $1_Y\r f_*f^*1_Y$,
using \rref{edd1}, and composing with the counit  of the adjunction $?^*$, $?_\sharp$.

We see now that the notation \rref{edd2}, \rref{edd3} is not really justified in terms
of parametrized $K$-modules: in effect, it is ``one level below in terms of 
$2$-category theory'', and fits more with the base change functors applied to
vector bundles which represent classes in the cohomology groups involved in
\rref{edd2}, \rref{edd3}. This makes selecting notation for the corresponding
maps in cohomology precarious.   As a compromise between possibly contradicting
allusions, we use $f_*$ for the induced map in homology, and $f^!$ for the transfer.
Thus, $f_*$ (in twisted $K$-homology) is induced by the adjunction counit
$$f_\sharp f^*\tau \r \tau,$$
$f^!$ (in homology) is induced by the adjunction unit
$$\tau\r f_*f^*\tau,$$
together with \rref{edd1}.

We should note that in this paper, we also use infinite extensions of these duality results,
allowed by the work of Cohen and Klein \cite{ckl}.

\end{document}